\documentclass[reqno,11pt]{amsart}
\usepackage{amsmath,amsfonts,amssymb,amsxtra,latexsym,amscd,enumerate,amsthm,verbatim}

\usepackage{graphicx}
\usepackage{xcolor}
\usepackage{tikz}
\usetikzlibrary{positioning, arrows.meta}

\usepackage{hyperref}

\usepackage[margin=1.40in]{geometry}
\numberwithin{equation}{section}

\newcommand{\R}{\mathbb{R}}

\newcommand{\C}{\mathbb{C}}
\newcommand{\Z}{\mathbb{Z}}

\newcommand{\wPsi}{\widetilde{\Psi}}
\newcommand{\wOmega}{\widetilde{\Omega}}

\newtheorem{theorem}{Theorem}[section]

\newtheorem{proposition}[theorem]{Proposition}

\newtheorem{remark}[theorem]{Remark}
\newtheorem{conjecture}[theorem]{Conjecture}

\newtheorem{numresult}[theorem]{Main Numerical Result}

\begin{document}

\title[Swirl-free unstable solutions of the Navier--Stokes equation]{On the non-uniqueness of solutions of the axi-symmetric swirl-free Navier--Stokes equations, I}

\author{A. D. Ionescu}
\address{Princeton University}
\email{aionescu@math.princeton.edu}
\author{H. Jia}
\address{University of Minnesota--Twin Cities}
\email{jia@umn.edu}
\author{S. Palasek}
\address{Institute for Advanced Study}
\email{palasek@ias.edu}

\begin{abstract}
In this paper we construct numerically a new class of unstable self-similar solutions of the incompressible Navier--Stokes equations in $\mathbb{R}^3$. Our solutions are axially symmetric and homogeneous of degree $-1$ at $\infty$, and are unstable in the sense that the linearization around these solutions contains unstable modes. Solutions of this type have been discovered numerically in \cite{GuSv} and \cite{HoWaYa}, and have applications to proving non-uniqueness results. 

The main novelty in this paper is that we discover the existence of  such solutions in the space of axially symmetric swirl-free (ASSF) vector fields. These approximate solutions are defined on all of $\mathbb R^3$ and achieve global pointwise residuals of order $10^{-10}$. We discuss the numerical construction of these solutions in detail, as well as their relevance to the problem of non-uniqueness of solutions of the incompressible Navier--Stokes equations in 3D, in the space of ASSF solutions.  
\end{abstract}

\maketitle

\setcounter{tocdepth}{1}

\tableofcontents

\setcounter{tocdepth}{1}

\section{Introduction}

We consider solutions $u$ of the incompressible Navier--Stokes equation in $\mathbb{R}^3$,
\begin{equation}\label{Nav1}
\partial_tu+u\cdot\nabla u+\nabla p-\Delta u=0,\qquad \nabla\cdot u=0.
\end{equation}
In 1934, Leray \cite{Leray} pioneered the analysis of \eqref{Nav1}, showing local well-posedness of strong solutions from regular initial data and global existence of weak solutions in the energy space $L_t^\infty L_x^2\cap L_t^2H_x^1$. Two fundamental open problems were raised: whether classical solutions can develop finite time singularities, and whether Leray weak solutions are unique.

The importance of Leray weak solutions stems, in part, from the weak-strong uniqueness phenomenon: if two Leray weak solutions have the same data at $t=0$, then they continue to agree as long as either solution is regular. Consequently, from smooth data, Leray solutions can violate uniqueness only after a blow-up has occurred. A natural possibility would be that some initially smooth solutions may develop a finite-time singularity and then split into several Leray weak solutions. Rigorously establishing such a scenario remains, of course, far out of reach. Instead, there has been significant effort to demonstrate non-uniqueness of Leray solutions from \emph{singular} initial data, ideally in sharp regularity classes. 

\subsection{The main conjecture} Our goal in this paper is to provide strong numerical evidence and a theoretical framework for proving the following conjecture:

\begin{conjecture}\label{IntroConjecture}
(i) There are two different solutions 
\begin{equation}\label{pr9}
u_1, u_2\in C([0,1]:H^\alpha_x)\cap L^2_tH^{\alpha+1}_x\cap L^\infty_tL^{3,\infty}_x,\qquad \text{ for any }\alpha\in[0,1/2),
\end{equation}
of the incompressible Navier--Stokes equations \eqref{Nav1} in $\R^3$, with the same initial data $u_1(0)=u_2(0)=u_0\in H^\alpha\cap L^{3,\infty}$, $\alpha\in[0,1/2)$. These solutions are smooth in $(0,1]\times\R^3$, satisfy $t^{1/2}\|u_j(t)\|_{L^\infty}\lesssim 1$, $j\in\{1,2\}$, and satisfy the energy identity 
\begin{equation}\label{pr10}
\frac{1}{2}\|u_j(t_0)\|_{L^2}^2=\frac{1}{2}\|u_j(t_1)\|_{L^2}^2+\int_{t_0}^{t_1}\|\nabla u_j(s)\|_{L^2}^2\,ds
\end{equation}
for any $t_0\leq t_1\in[0,1]$ and $j\in\{1,2\}$.

(ii) (Strong form) The solutions $u_1, u_2$ are axially symmetric and swirl-free.
\end{conjecture}

This conjecture would provide a definitive answer to the question of non-uniqueness of solutions of the incompressible NSE in 3D, improving significantly on all the earlier non-uniqueness results in \cite{BuVi, AlBrCo, CoPa, HoWaYa, ChDaPa} (see the discussion below). The resulting solutions are Leray solutions that just miss being in the Ladyzhenskaya--Prodi--Serrin regularity class, but are in the critical space $L^\infty_tL^{3,\infty}_x\cap L^{2,\infty}_tL^\infty_x$ and in optimal Sobolev spaces. 

Remarkably, proving Conjecture \ref{IntroConjecture} (ii) in the class of axi-symmetric swirl-free vector fields (as suggested by our numerics) would have the implication that the non-uniqueness question is, in fact, completely decoupled from other outstanding questions concerning regularity of solutions of the 3D incompressible NSE. 

Indeed, the incompressible Navier--Stokes (or even Euler) equations in 3D in the case of ASSF solutions are well-known to be ``regular" equations, similar to the 2D equations, for which large-data global regularity holds (see the classical work of Ladyzhenskaya \cite{La} and Ukhovskii--Yudovich \cite{UkYu}, and the more recent paper of Gallay--\v Sver\'{a}k \cite{GaSv} for a scale-invariant global well-posedness result). Conjecture \ref{IntroConjecture} (ii) would show that sharp non-uniqueness still holds for this regular equation, and therefore is not related to unresolved issues such as singularity formation of 3D incompressible Navier--Stokes flows.

\subsubsection{Previous work on Navier--Stokes non-uniqueness}

The first construction of non-unique weak solutions of the Navier--Stokes equations was given by Buckmaster--Vicol~\cite{BuVi} in the space $u\in C_tH^\epsilon(\mathbb R^3)$ using convex integration. In the dissipative setting, despite a great deal of work in this direction, these methods generally fail to produce counterexamples close to the sharp regularity, which in the $H^s$ scale in dimension 3 is $s<1/2$.

The problem can be reframed in the context of critical or slightly super-critical non-uniqueness results relative to the Ladyzhenskaya--Prodi--Serrin (LPS) criterion, which states that weak solutions are unique in the space $L_t^p([0,T];L_x^q(\mathbb R^d))$ where $p\in[2,\infty)$ and $2/p+d/q=1$. Cheskidov--Luo~\cite{ChLu1,ChLu2} used convex integration to prove non-uniqueness slightly below the endpoints $(p,q)=(2,\infty)$ for all $d\geq2$ and $(p,q)=(\infty,2)$ when $d=2$; meanwhile, the rest of the LPS scale seems quite out of reach of these methods.

A different method based on a discrete energy cascade was introduced by Coiculescu--Palasek in \cite{CoPa} to produce the first examples of non-uniqueness at the critical regularity. These solutions are smooth after $t=0$ and lie in the critical space $L_t^\infty BMO_x^{-1}\cap L_t^{2,\infty}L_x^\infty$ in 3D, subsequently extended to dimension two and the space $L_t^\infty W_x^{-1,\infty}$ in~\cite{ChDaPa}.

The solutions described above are not Leray solutions and do not satisfy the weak-strong uniqueness principle.  Jia--\v Sver\'{a}k \cite{JiSv2} and Guillod--\v Sver\'{a}k \cite{GuSv} proposed an approach towards proving sharp non-uniqueness of Leray weak solutions, based on self-similar profiles and bifurcation. Using this method Albritton--Bru\'{e}--Colombo \cite{AlBrCo} first proved non-uniqueness of Leray solutions of the incompressible NSE with singular forcing. Very recently, Hou--Wang--Yang \cite{HoWaYa} announced a computer-assisted proof of non-uniqueness of the classical Leray solutions of the incompressible NSE without forcing. 
\smallskip

The framework we propose here for the proof of Conjecture \ref{IntroConjecture} is based on this method, with two essential improvements: 

(1) We find numerically a new class of steady self-similar profiles, which are asymptotically homogeneous of degree $-1$, linearly unstable, axially symmetric, and swirl-free. The main novelty here is the swirl-free condition; the other conditions were already part of the numerical investigations of \cite{GuSv} and \cite{HoWaYa}. Instability in our case appears to be of a different nature than in these earlier works, since it is not generated by the swirl component. For contrast, the unstable profiles constructed numerically in \cite{GuSv} are pure-swirl at infinity. The unstable profiles proposed in \cite{HoWaYa} are not provided explicitly in the paper, but the pictures suggest that they contain all three velocity fields.

(2) We prove stronger estimates on the resulting non-unique solutions of the Navier--Stokes equations (see Theorem \ref{MainThm} below), in optimal Sobolev spaces and critical LPS-type spaces (which imply in particular the Leray energy inequality in the strong form \eqref{pr10}). 

\subsection{The Navier--Stokes equations in self-similar variables}\label{maincal}
Our main idea to prove Conjecture \ref{IntroConjecture} is to exploit the link between non-uniqueness and self-similar solutions of the incompressible NSE. We are looking for scale-invariant solutions $u$ of the form
\begin{equation}\label{Nav2}
u(t,\mathbf{x})=\frac{1}{t^{1/2}}U\Big(\frac{\mathbf{x}}{t^{1/2}},\log t\Big),\qquad p(t,x)=\frac{1}{t}P\Big(\frac{\mathbf{x}}{t^{1/2}},\log t\Big).
\end{equation}
The equation \eqref{Nav1} becomes, in terms of the self-similar profiles $U,P$,
\begin{equation}\label{Nav3}
\partial_tU-\Delta U-\frac{1}{2}\mathbf{x}\cdot\nabla U-\frac{1}{2}U+U\cdot\nabla U+\nabla P=0,\qquad \nabla\cdot U=0.
\end{equation}

Steady self-similar solutions of the system \eqref{Nav3} have been investigated extensively, by many authors, and we refer to \cite{JiSv1} (or to \cite{AlGuKoRe} and \cite{GuLiXi} for the 2D case) for a longer discussion of the problem, construction of general asymptotically homogeneous solutions of degree $-1$, and more references. Our main interest in these solutions is due to their relevance to the classical problem of non-uniqueness of solutions of the incompressible Navier--Stokes equations in 3D. This link was discovered by Jia--\v Sver\'{a}k \cite{JiSv2}, and expanded by Guillod--\v Sver\'{a}k \cite{GuSv}, Albritton--Bru\'{e}--Colombo \cite{AlBrCo} and Hou--Wang--Yang \cite{HoWaYa}. We have the following precise theorem:

\begin{theorem}\label{MainThm}
Assume that there are $C^2$ divergence-free vector fields $(U,\widetilde{U})$ and a number $\lambda$ with $\Re\lambda>0$ satisfying the identities 
\begin{equation}\label{pr6}
\begin{split}
&\Delta U+\frac{1}{2}\mathbf{x}\cdot\nabla U+\frac{1}{2}U-\Pi(U\cdot\nabla U)=0,\\
&\Delta \widetilde{U}+\frac{1}{2}\mathbf{x}\cdot\nabla \widetilde{U}+\frac{1}{2}\widetilde{U}-\Pi\big(U\cdot\nabla \widetilde{U}+\widetilde{U}\cdot\nabla U\big)=\lambda\widetilde{U},
\end{split}
\end{equation}
where $\Pi$ denotes the Leray projection. Assume also that the vector fields $U,\widetilde{U}$ satisfy the bounds 
\begin{equation}\label{pr7}
|U(x)-U_0(x)|\lesssim \langle x\rangle^{-2},\qquad |\widetilde{U}(x)|\lesssim \langle x\rangle^{-2},
\end{equation}
for a divergence-free $-1$-homogeneous vector field $U_0\in C^\infty(\R^3\backslash\{0\})$. 

Then Conjecture~\ref{IntroConjecture} (i) holds. In addition, if the vector fields $U$ and $\widetilde{U}$ are axially symmetric and swirl-free then Conjecture \ref{IntroConjecture} (ii) holds.
\end{theorem}

We will provide a complete proof of Theorem \ref{MainThm} in section \ref{thimpl} below. Results of this type have been known and used in the past, in \cite{JiSv2}, \cite{AlBrCo}, and \cite{HoWaYa}. Here we prove stronger conclusions on the non-unique solutions, consistent with the claim \eqref{pr9} and part (ii) of Conjecture \ref{IntroConjecture}.

\subsection{Axi-symmetric swirl-free solutions}\label{axisym}  Our main goal in this paper is to present a new class of steady solutions of the system \eqref{pr6}, which are asymptotically homogeneous of degree $-1$, linearly unstable, axially symmetric, and swirl-free. 

We work in cylindrical coordinates $\mathbf{x}=(r\cos\theta,r\sin\theta,z)$ and define the vector fields
\begin{equation}\label{Nav4}
e_r:={}^t(\cos\theta,\sin\theta,0),\qquad e_\theta:={}^t(-\sin\theta,\cos\theta,0),\qquad e_z:={}^t(0,0,1)
\end{equation}
and the functions $U_r$, $U_\theta$, $U_z$ by
\begin{equation}\label{Nav5}
U=U_re_r+U_\theta e_\theta+U_ze_z.
\end{equation}
In the axially symmetric case $\partial_\theta U_r=\partial_\theta U_\theta=\partial_\theta U_z=\partial_\theta P=0$, and the equations \eqref{Nav3} are equivalent to
\begin{equation}\label{Nav7}
\begin{split}
\Big(\widetilde{\Delta}-\frac{1}{r^2}\Big)U_r+\frac{1}{2}\big(r\partial_r+z\partial_z+1\big)U_r-\big[U_r\partial_r+U_z\partial_z\big]U_r+\frac{1}{r}U_\theta^2-\partial_rP&=\partial_tU_r,\\
\Big(\widetilde{\Delta}-\frac{1}{r^2}\Big)U_\theta+\frac{1}{2}\big(r\partial_r+z\partial_z+1\big)U_\theta-\big[U_r\partial_r+U_z\partial_z\big]U_\theta-\frac{1}{r}U_\theta U_r&=\partial_tU_\theta,\\
\widetilde{\Delta} U_z+\frac{1}{2}\big(r\partial_r+z\partial_z+1\big)U_z-\big[U_r\partial_r+U_z\partial_z\big]U_z-\partial_zP&=\partial_tU_z,
\end{split}
\end{equation}
where $\widetilde{\Delta}:=\partial^2_r+(1/r)\partial_r+\partial^2_z$, while the incompressibility condition becomes
\begin{equation}\label{Nav8}
\partial_rU_r+\frac{1}{r}U_r+\partial_z U_z=0.
\end{equation}

We are in fact looking for swirl-free steady solutions, which correspond to $U_\theta\equiv 0$. The variable $U_r$ is not smooth in $r$ due to a coordinate singularity at $r=0$. To work with smooth variables we define
\begin{equation}\label{Nav8.1}
\Omega:=\frac{\partial_zU_r-\partial_rU_z}{r}.
\end{equation}
We can further simplify the equations by introducing the stream function $\Psi$ defined by
\begin{equation}\label{Nav9.1}
\mathcal{L}\Psi=\Omega,\qquad \mathcal{L}:=\partial_r^2+\partial_z^2+\frac{3}{r}\partial_r.
\end{equation}
It follows from \eqref{Nav8}--\eqref{Nav8.1} that
\begin{equation}\label{Nav8.3}
U_r=r\partial_z\Psi,\qquad U_z=-r\partial_r\Psi-2\Psi.
\end{equation}

To summarize, in terms of the new variables $\Psi,\Omega:\R^2\to\R$, the steady self-similar Navier--Stokes system \eqref{Nav7} is equivalent to the system
\begin{equation}\label{sf1}
\begin{split}
&\mathcal{L}\Psi-\Omega=0,\\
&\mathcal{L}\Omega+\frac{1}{2}\big(r\partial_r+z\partial_z+3\big)\Omega-r\partial_z\Psi\partial_r\Omega+r\partial_r\Psi\partial_z\Omega+2\Psi\partial_z\Omega=0.
\end{split}
\end{equation}
We are interested in solutions of this system which are sufficiently smooth in $r,z$ (say $C^2$), even in $r$, decay at suitable rates at infinity, i.e.
\begin{equation*}
\langle r,z\rangle|\Psi(r,z)|+\langle r,z\rangle^3|\Omega(r,z)|\lesssim 1,
\end{equation*}
and are linearly unstable  in the sense that 
\begin{equation}\label{Blt2}
\begin{split}
&\mathcal{L}\wPsi-\wOmega=0,\\
&\mathcal{L}\wOmega+\frac{1}{2}\big(r\partial_r+z\partial_z+3\big)\wOmega+\big\{r\partial_r\Psi\partial_z\wOmega+r\partial_z\Omega\partial_r\wPsi\\
&\qquad\qquad-r\partial_z\Psi\partial_r\wOmega-r\partial_r\Omega\partial_z\wPsi+2\Psi\partial_z\wOmega+2\partial_z\Omega\wPsi\big\}-\lambda \wOmega=0
\end{split}
\end{equation}
for some nontrivial $C^2$ functions $(\wPsi,\wOmega)$ that are even in $r$ and decay at infinity, and some $\lambda\in\C$ with $\Re\lambda>0$. 

For $k\in\Z_+$ let $C^k_{e,o}(\R^2)$ (respectively $C^k_{e,e}(\R^2)$) denote the spaces of $C^k$ functions on $\R^2$ which are even in $r$ and odd in $z$ (respectively even in $r$ and even in $z$). 

Our main numerical result in this paper is the following:

\begin{numresult}\label{MainNumRes} (i) There are approximate global solutions $\Psi, \Omega\in C^2_{e,o}(\R^2)$ and $\wPsi, \wOmega\in C^2_{e,e}(\R^2)$, normalized with $\sup_{r,z\in\R^2}|\wPsi(r,z), \wOmega(r,z)|=1$, and an approximate eigenvalue 
\begin{equation}\label{MainNum1}
\lambda_\ast=0.235059597921
\end{equation}
such that the global pointwise residual bounds
\begin{equation}\label{MainNum2}
\begin{split}
&\|\mathrm{Res}(\Psi)\|_{L^\infty(\R^2)}\leq 3\times 10^{-10},\,\,\qquad \|\mathrm{Res}(\Omega)\|_{L^\infty(\R^2)}\leq 3\times 10^{-10},\\
&\|\mathrm{Res}(\wPsi)\|_{L^\infty(\R^2)}\leq 8\times 10^{-11},\qquad\,\,\|\mathrm{Res}(\wOmega)\|_{L^\infty(\R^2)}\leq 3\times 10^{-10}
\end{split}
\end{equation}
hold in $\R^2$, where $\mathrm{Res}(\Psi)$, $\mathrm{Res}(\Omega)$, $\mathrm{Res}(\wPsi)$, $\mathrm{Res}(\wOmega)$ denote the expressions in the left-hand side of the equations \eqref{sf1}--\eqref{Blt2}.

(ii) The functions $\Psi,\Omega\in C^2_{e,o}(\R^2)$ satisfy the exact identities
\begin{equation}\label{MainNum3}
\begin{split}
&\Psi(r,z)=\frac{\sigma z}{r^2+z^2}+\frac{\sigma^2z(3r^2-7z^2)}{3(r^2+z^2)^3}-\frac{6\sigma z}{(r^2+z^2)^2},\\
&\Omega(r,z)=\frac{-6\sigma z}{(r^2+z^2)^2}+\frac{12\sigma^2z(3z^2-2r^2)+24\sigma z(r^2+z^2)}{(r^2+z^2)^4},
\end{split}
\end{equation}
for any $(r,z)\in\R^2\setminus[-L',L']^2$, where $\sigma=130$ and $L'=751$. The functions $\wPsi,\wOmega\in C^2_{e,e}(\R^2)$ vanish identically in $\R^2\setminus[-L',L']^2$. 

The residuals $\mathrm{Res}(\Psi)$, $\mathrm{Res}(\wPsi)$, and $\mathrm{Res}(\wOmega)$ also vanish in $\R^2\setminus[-L',L']^2$, while
\begin{equation}\label{asympb}
|\mathrm{Res}(\Omega)(r,z)|\leq \frac{500\sigma^3|z|}{(r^2+z^2)^4}\leq 2\times 10^{-11}\qquad\text{ for }(r,z)\in\R^2\setminus[-L',L']^2.
\end{equation}
\end{numresult}

We emphasize that there is a parity change in $z$ of the unstable eigenvector $(\wPsi,\wOmega)$ compared to the base solution $(\Psi,\Omega)$. This is a key feature of the problem, and it is present in the earlier constructions in \cite{GuSv} and \cite{HoWaYa}. See Figures \ref{fig:main_functions} and \ref{fig:velocities} for plots of the main variables $(\Psi,\Omega),(\widetilde{\Psi},\widetilde{\Omega})$ and of the corresponding velocity fields $(U_r, U_z), (\widetilde{U}_r, \widetilde{U}_z)$ restricted to the region $[0,50]^2$. 

\begin{figure}[ht]
\centering
\includegraphics[width=\textwidth]{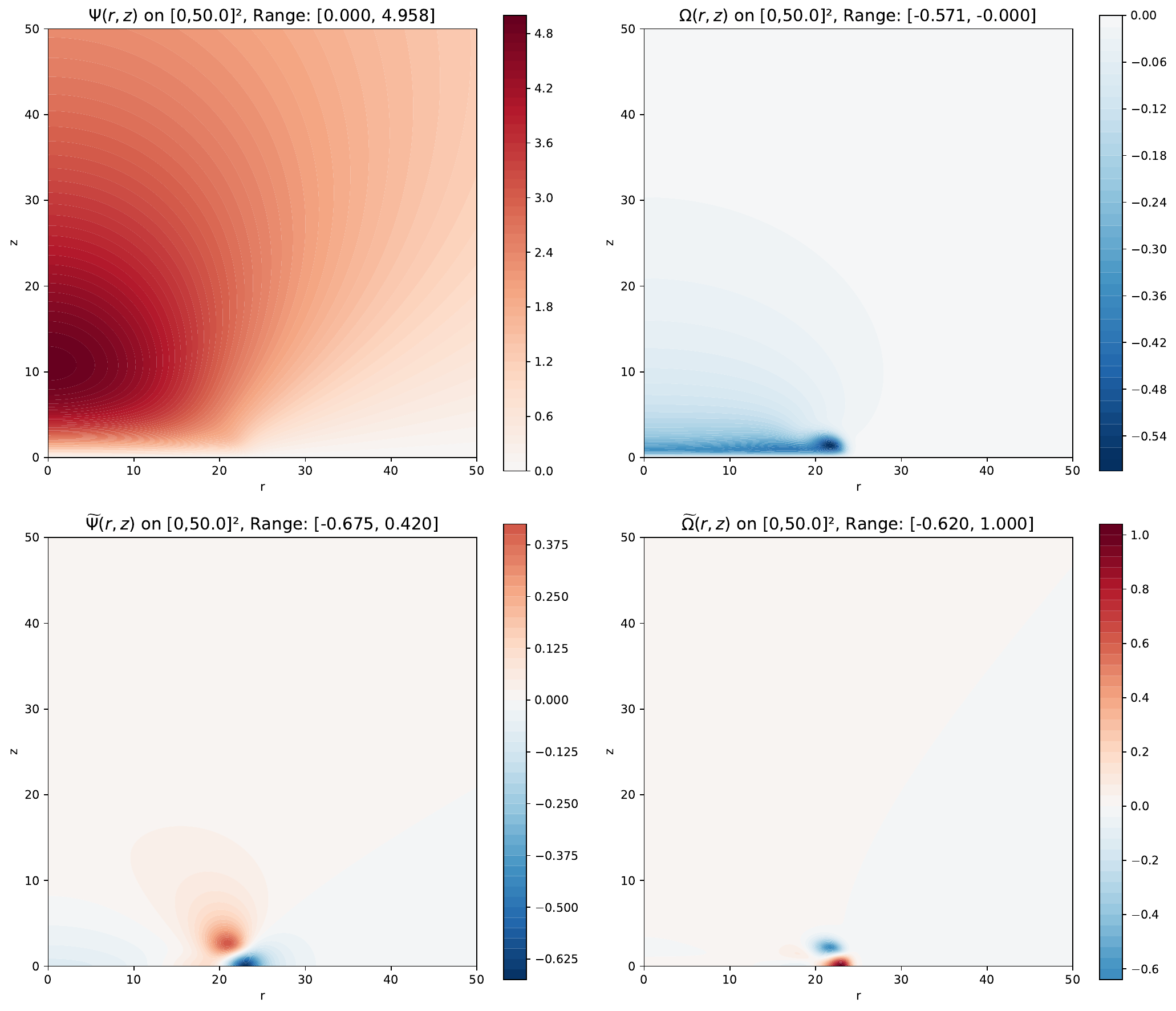}
\caption{Plots of the solutions $(\Psi,\Omega)$ and $(\widetilde{\Psi},\widetilde{\Omega})$ described in the Main Numerical Result \ref{MainNumRes} restricted to domain $(r,z)\in[0,50]^2$.}
\label{fig:main_functions}
\end{figure}

\begin{remark}
Our solutions $(\Psi,\Omega)$, $(\widetilde{\Psi},\widetilde{\Omega})$ are represented as B-splines of degree 10 with about 148,000 degrees of freedom for each variable. We use global $L^\infty$ norms evaluated on dense grids in \eqref{MainNum2} to validate these solutions (see also Figure \ref{fig:global_res} for more detailed residual heatmaps of the solutions $(\Psi,\Omega)$ and $(\widetilde{\Psi},\widetilde{\Omega})$). One should think of these norms as proxy-type norms; other norms, such as $L^2$ norms with various weights, can be used as well. However, verifying pointwise $L^\infty$ bounds on a dense, independent grid provides a particularly rigorous form of validation, as it guarantees that the residuals are uniformly controlled across the entire physical domain. 

One could potentially reduce the residuals further, for example by increasing the number of basis functions of the B-splines, further optimizing the adaptive knot distribution, or using more precise asymptotic expansions (as discussed in subsection \ref{asym_extra}). Nevertheless, we note that the pointwise residuals we obtain here are global in space and are already better than the residuals reported in recent numerical investigations and computer-assisted proofs of similar 2D fluid equations, including the work on the Navier--Stokes equations in \cite{GuSv} and \cite{HoWaYa}, as well as the rigorous numerical bounds for the Euler equations established in \cite{ChenHou2022, ChenHou2023}. Furthermore, our pointwise error bounds also exceed the precision achieved by modern machine-learning-based methods, such as PINNs applied to the Boussinesq equation in \cite{PINN}.
\end{remark}

\begin{figure}[ht]
\centering
\includegraphics[width=\textwidth]{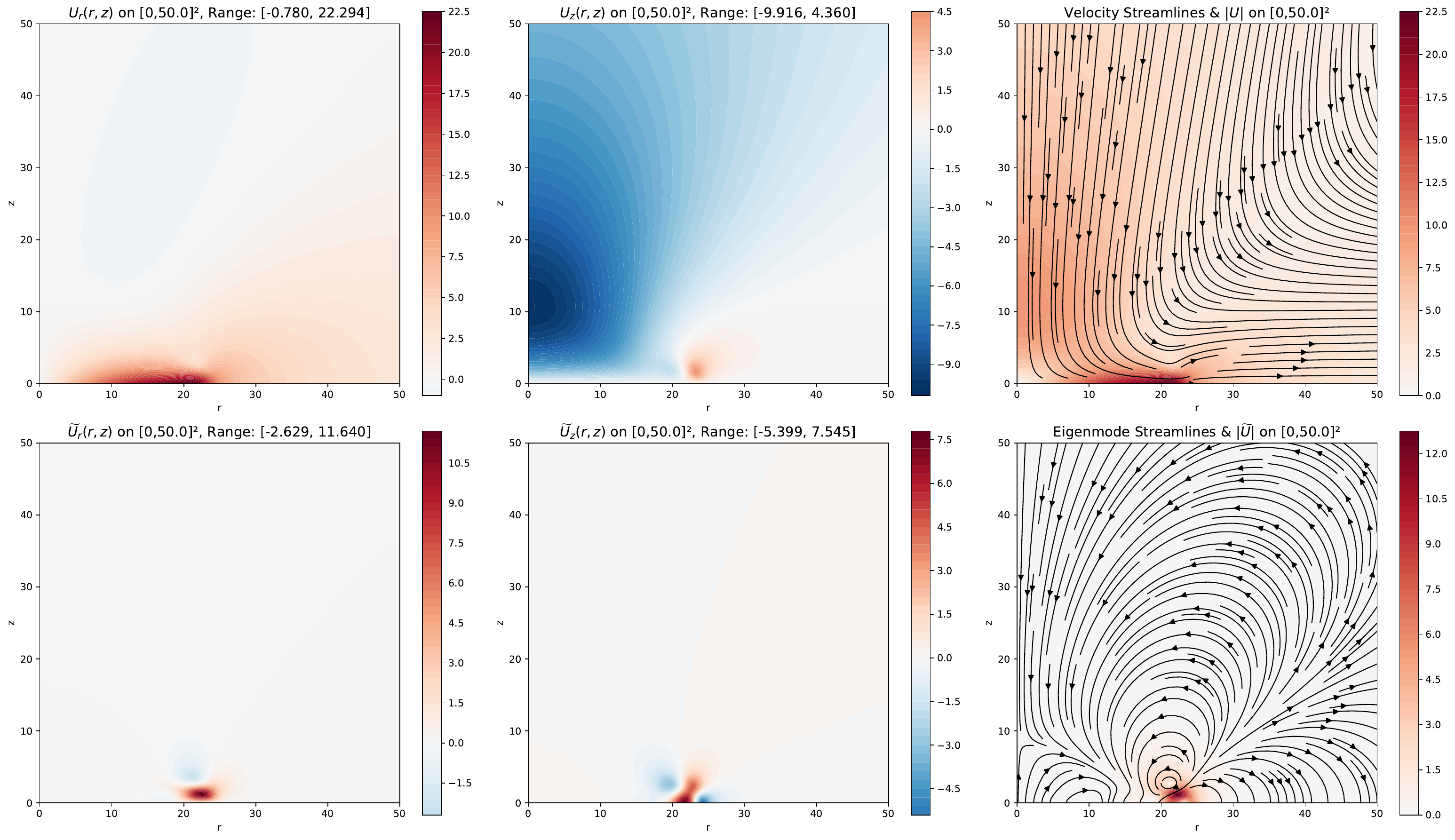}
\caption{Plots of the physical velocity fields $(U_r, U_z)$ and $(\widetilde{U}_r, \widetilde{U}_z)$ corresponding to the solutions $(\Omega,\Psi)$ and $(\widetilde{\Omega}, \widetilde{\Psi})$, restricted to domain $(r,z)\in[0,50]^2$.}
\label{fig:velocities}
\end{figure}

This numerical result strongly indicates that the approximate solutions constructed above can be upgraded to exact solutions. Specifically, the main objective of our program is to prove rigorously that there exist exact solutions $(\Psi,\Omega)\in C^2_{e,o}(\R^2)\times C^2_{e,o}(\R^2)$ and $(\lambda, \wPsi, \wOmega)\in[0.2,0.3]\times C^2_{e,e}\times C^2_{e,e}$ of systems \eqref{sf1} and \eqref{Blt2}
in $\R^2$, with $\sigma=130$ and the precise asymptotics
\begin{equation}\label{pr2}
\begin{split}
&\Psi(r,z)=\frac{\sigma z}{r^2+z^2}+O((r^2+z^2)^{-1}),\qquad \Omega(r,z)=\frac{-6\sigma z}{(r^2+z^2)^2}+O((r^2+z^2)^{-2}),\\
&\wPsi(r,z)=O((r^2+z^2)^{-1}),\qquad \wOmega(r,z)=O((r^2+z^2)^{-2}).
\end{split}
\end{equation}

Our main point is that establishing this will complete the proof of Conjecture \ref{IntroConjecture}, since the hypothesis \eqref{pr6}--\eqref{pr7} of Theorem \ref{MainThm} at the level of the velocity fields follows easily from the existence of these exact vorticity-stream function profiles. We will provide the rigorous, computer-assisted perturbative proof of this statement in the second part of this work.

\subsection{Organization} The rest of this paper is organized as follows: in section \ref{numerical} we describe our numerical construction in detail. In fact, we provide numerical evidence showing that there is a curve of smooth solutions $(\Psi,\Omega)=(\Psi_\sigma, \Omega_\sigma)$ of the system \eqref{sf1}, even in $r$ and odd in $z$, with the simple asymptotics at $\infty$
\begin{equation}\label{sf2}
\Psi(r,z)=\frac{\sigma z}{r^2+z^2}+O((r^2+z^2)^{-1}),\qquad \Omega(r,z)=\frac{-6\sigma z}{(r^2+z^2)^2}+O((r^2+z^2)^{-2}),
\end{equation}
where $\sigma\in\R$ is a parameter, which become unstable for $\sigma$ sufficiently large. The solutions we describe above correspond to the specific value $\sigma=130$ and are refined with the more precise asymptotic conditions \eqref{MainNum3}, which is needed in order to get better global residuals in \eqref{MainNum2}.

In section \ref{thimpl} we provide a proof of Theorem \ref{MainThm},  thus showing that the existence and regularity of exact unstable solutions of the system \eqref{sf1}--\eqref{Blt2} would lead to the definitive non-uniqueness result stated in Conjecture \ref{IntroConjecture}.

In section~\ref{sec:4} we discuss the stability of our numerical construction, as well as the derivation of the refined asymptotic expansion \eqref{MainNum3} at infinity.

\subsection{Acknowledgements}

A. D. I. was supported in part by NSF-FRG grant DMS-2245228 and by NSF-UEFISCDI grant DMS-2407694. S. P. was supported by NSF Grant No. DMS-2424441. H. J. was supported in part by NSF-FRG grant DMS-2245021 and NSF grant DMS-2453270. 

The authors acknowledge that this work was performed using the Princeton Research Computing resources, a consortium including the Princeton Institute for Computational Science and Engineering (PICSciE) and Research Computing at Princeton University. The authors would also like to thank Javier G\'{o}mez-Serrano for useful discussions related to this work. 

Finally, the authors acknowledge the use of AI assistants to help draft portions of the Python codes used for the numerical simulations. The numerical analysis, including the Python codes, the changes of variables, the analytical Jacobian formulations, the specialized collocation grids, and the dense-grid validation, were designed and rigorously verified by the authors.

\section{Numerical construction and results}\label{numerical}

All Python codes used to perform the numerical constructions, eigenvalue tracking, and residual validations described in this section are publicly available on GitHub at
\begin{center}
\url{https://github.com/alexionescu2/Leray_Uniqueness}.
\end{center}

The main idea of the construction of the solutions $(\Psi,\Omega)$ and $(\lambda,\wPsi, \wOmega)$ in the Main Numerical Result \ref{MainNumRes} is to first construct precise solutions in a large square $(r,z)\in [-L,L]^2$ and then glue them using cutoff functions to their asymptotic data. There are two main issues to keep in mind, both of them having direct implications on the residual bounds \eqref{MainNum2}: (1) the numerical solutions in the compact set $[-L,L]^2$ have to satisfy the parity conditions exactly and have small residuals, and (2) the square $[-L,L]^2$ has to be large enough in order to lead to small residuals in the gluing region. We proceed in several steps:

\subsection{Step 1: renormalization} To construct precise global solutions with prescribed asymptotic states we need to construct precise numerical solutions in very large Euclidean domains. One could, in principle, compactify $\R^2$ by a change of variables, and reduce the problem to the unit square (or the unit ball), but the Laplacian becomes a degenerate operator in compactified coordinates, leading to numerical instabilities. 

We approach this issue differently: we make changes of variables of the form $u\approx (1+r^2)^{1/p}$, $v\approx (1+z^2)^{1/p}$, for suitable exponents $p\geq 2$, and construct first numerical solutions in the $(u,v)$ space. We use the change of variables to also encode the important symmetries of the nonlinear solutions $(\Psi,\Omega)$ (even in $r$ and odd in $z$) and of the unstable eigenfunctions $(\wPsi, \wOmega)$ (even in $r$ and even in $z$). Precisely, we define 
\begin{equation}\label{var6}
\begin{split}
&\Psi(r,z):=z\Psi_2(u,v),\qquad \Omega(r,z):=z\Omega_2(u,v),\\
&\wPsi(r,z):=\wPsi_2(u,v),\qquad\,\,\,\,\wOmega(r,z):=\wOmega_2(u,v),\\
& u=f(r^2), \qquad v=f(z^2),
\end{split}
\end{equation}
where $p\in\{4,5,6\}$ and 
\begin{equation}\label{var14}
f(x)=(1+12x)^{1/p}-1.
\end{equation}

\begin{figure}[ht]
\centering
\includegraphics[width=0.8\textwidth]{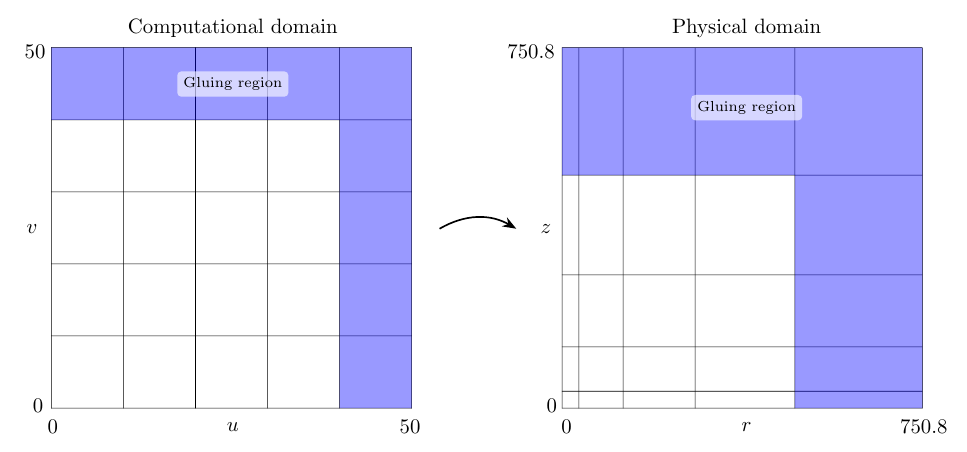}
\caption{The coordinate mapping $(u,v)\mapsto(r,z)$ for $(A,p)=(50,4)$. Left: a uniform grid in the computational domain $(u,v)\in[0,50]^2$. Right: the image of the same grid lines in the physical domain $(r,z)$, showing the concentration of resolution near the origin.}
\label{fig:coordmap}
\end{figure}

In terms of the new variables, the equations \eqref{sf1} are equivalent to the system
\begin{equation}\label{var17}
\begin{split}
&\mathcal{L}_{uv}\Psi_2-\Omega_2=0,\\
&\mathcal{L}_{uv}\Omega_2+\big\{K(u)\partial_u+K(v)\partial_v+2\big\}\Omega_2+2\big\{2K(u)K(v)\big(\partial_u\Psi_2\partial_v\Omega_2-\partial_v\Psi_2\partial_u\Omega_2\big)\\
&\qquad\qquad-K(u)(\Psi_2\partial_u\Omega_2-\Omega_2\partial_u\Psi_2)+2K(v)\Psi_2\partial_v\Omega_2+\Psi_2\Omega_2\big\}=0
\end{split}
\end{equation}
where
\begin{equation}\label{var16}
\begin{split}
&\mathcal{L}_{uv}:=\frac{48}{p}\Big\{\frac{K(u)}{(u+1)^{p-1}}\partial_u^2+\frac{(p+1)(u+1)^p+p-1}{p(u+1)^{2p-1}}\partial_u\\
&\qquad\qquad\qquad+\frac{K(v)}{(v+1)^{p-1}}\partial_v^2+\frac{(p/2+1)(v+1)^p+p-1}{p(v+1)^{2p-1}}\partial_v\Big\},\\
&K(x):=\frac{(x+1)^p-1}{p(x+1)^{p-1}}.
\end{split}
\end{equation}

We think of the functions $\Psi_2,\Omega_2$ as defined in large computational domains $(u,v)\in[0,A]^2$ and impose the boundary conditions (consistent with \eqref{sf2})
\begin{equation}\label{var13}
\Psi_2(u,v)=\frac{\sigma}{U_0(u)+U_0(v)},\qquad\Omega_2(u,v)=\frac{-6\sigma}{(U_0(u)+U_0(v))^2}\qquad\text{ if }v=A\text{ or }u=A,
\end{equation}
where $U_0(u)=[(u+1)^p-1]/12$. There are no boundary conditions on the other two sides of the square $[0,A]^2$, corresponding to $v=0$ or $u=0$, due to the choice of the variables \eqref{var6}. The parameter $\sigma$ is the continuation parameter that increases from $\sigma=0$ to $\sigma=160$. 

Using the changes of variables \eqref{var6} the linearized system \eqref{Blt2} becomes
\begin{equation}\label{Blt12}
\begin{split}
&\mathcal{L}'_{uv}\wPsi_2-\wOmega_2=0,\\
&\mathcal{L}'_{uv}\wOmega_2+\big(K(u)\partial_u+K(v)\partial_v+3/2\big)\wOmega_2+2\big\{2K(u)K(v)\big(\partial_u\Psi_2\partial_v\wOmega_2-\partial_v\Psi_2\partial_u\wOmega_2\big)\\
&\qquad-2K(u)K(v)\big(\partial_u\Omega_2\partial_v\wPsi_2-\partial_v\Omega_2\partial_u\wPsi_2\big)-K(u)(\Psi_2\partial_u\wOmega_2-\Omega_2\partial_u\wPsi_2)\\
&\qquad+2K(v)\big(\Psi_2\partial_v\wOmega_2+\partial_v\Omega_2\wPsi_2\big)+\Omega_2\wPsi_2\big\}-\lambda\wOmega_2=0,
\end{split}
\end{equation}
where $\lambda$ is the generalized eigenvalue to be found and
\begin{equation}\label{var16.2}
\begin{split}
&\mathcal{L}'_{uv}:=\frac{48}{p}\Big\{\frac{K(u)}{(u+1)^{p-1}}\partial_u^2+\frac{(p+1)(u+1)^p+p-1}{p(u+1)^{2p-1}}\partial_u\\
&\qquad\qquad\qquad+\frac{K(v)}{(v+1)^{p-1}}\partial_v^2+\frac{(1-p/2)(v+1)^p+p-1}{p(v+1)^{2p-1}}\partial_v\Big\}.
\end{split}
\end{equation}
Notice that there are small differences between the Laplacians $\mathcal{L}_{uv}$ in \eqref{var16} and $\mathcal{L}'_{uv}$ in \eqref{var16.2}, due to the difference in the changes of variables used for the nonlinear solutions and for the generalized eigenfunctions (coming mainly from the parity change).  

To summarize, we are first constructing high-precision solutions $(\Psi_2,\Omega_2)$ for the system \eqref{var17} and $(\lambda,\wPsi_2,\wOmega_2)$ for the system \eqref{Blt12} on squares $(u,v)\in[0,A]^2$. The changes of variables \eqref{var6} serve two main purposes: (1) they encode the desired symmetries of the functions $(\Psi,\Omega, \wPsi,\wOmega)$, and (2) they compress very large physical domains of the form $[-L,L]$, $L\approx 1000$, to computational domains of moderate size, for example $(A,p)=(50,4)$ or $(A,p)=(25,5)$ or $(A,p)=(15,6)$.

\subsection{Step 2: $\sigma$-continuation} For the sake of concreteness, we describe our analysis in the case $(A,p)=(50,4)$, which corresponds to the square domain $(r,z)\in[-L,L]^2$ in the physical space, where $L=\sqrt{[(A+1)^4-1]/12}\approx 750.8$. Similar results hold for other choices of the parameters $A$ and $p$, such as $(A,p)=(15,6)$ or $(A,p)=(25,4)$ or $(A,p)=(25,5)$.
\smallskip

{\bf{(1) Base state construction.}} We are searching for solutions $(\Psi_2,\Omega_2)=(\sigma/\kappa)(\Psi'_2,\Omega'_2)$ of the system \eqref{var17}--\eqref{var13}, where we fix the parameter $\kappa:=150$ mainly to work with unit-size fields throughout the $\sigma$-continuation. We initialize the numerical construction by solving the renormalized system at $\sigma=0$, in which case the system is linear. The unknowns $\Psi'_2$ and $\Omega'_2$ are represented using tensor products of  B-splines of degree 4 defined on a uniform grid in the computational domain $(u,v)\in[0,A]^2$. We use $N=400$ basis functions in each coordinate direction, resulting in a total problem size of $2N^2$. The discretized equations are enforced at the collocation points (the Greville abscissae) associated with the knot vectors.

\begin{figure}[ht]
\centering
\includegraphics[width=1.01\textwidth]{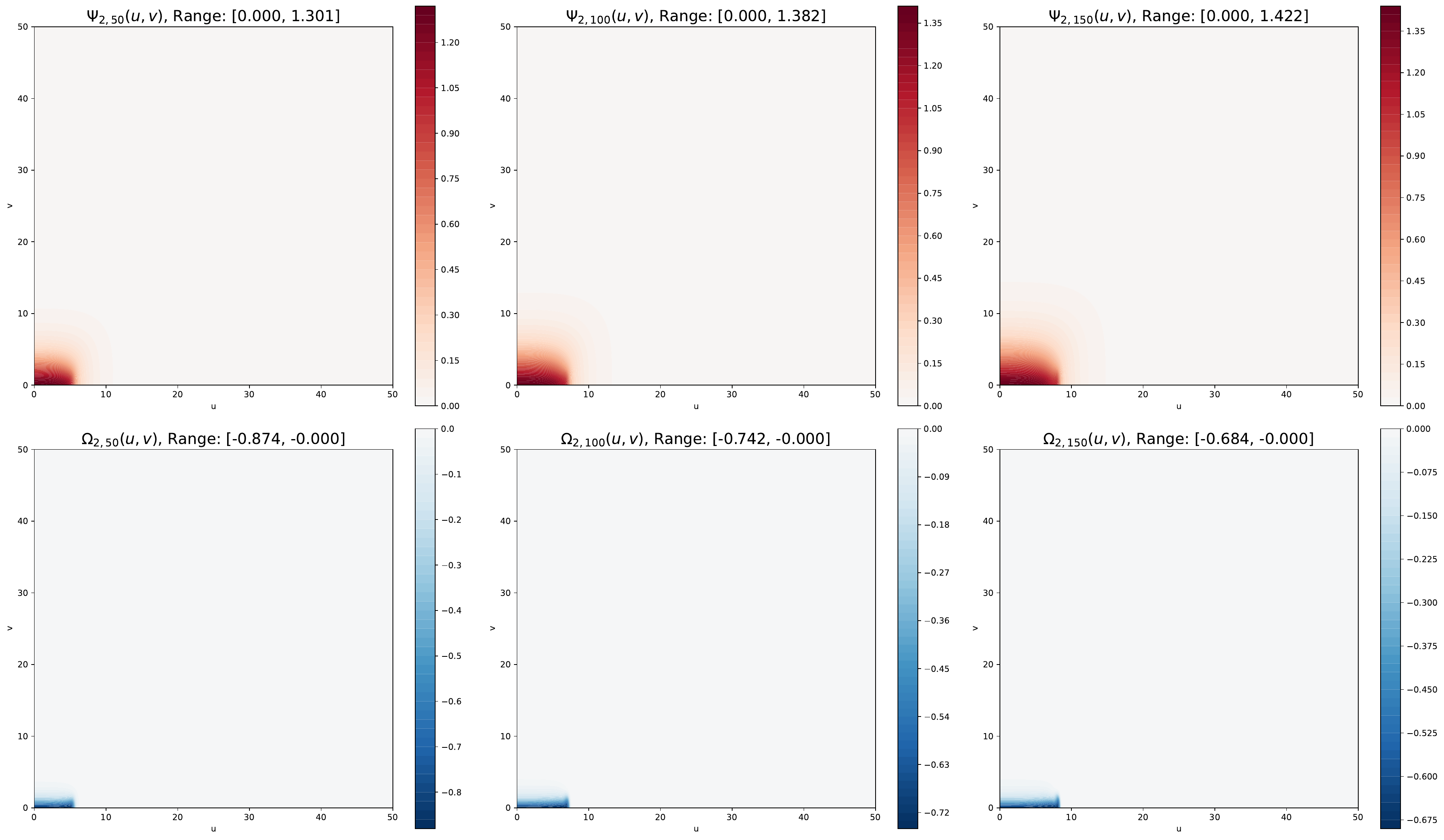}
\caption{Evolution of the nonlinear solution $(\sigma/\kappa)(\Psi'_2,\Omega'_2)$ on the computational domain $(u,v)\in[0,50]^2$, plotted for three values $\sigma=50,100,150$.}
\label{fig:evolution}
\end{figure}

We then employ natural parameter continuation with a linear secant predictor to advance in $\sigma$. The predictor extrapolates the initial guess for step $n+1$ using the tangent approximation derived from solutions at steps $n$ and $n-1$. At each step, the nonlinear system is solved using a Newton-Krylov method (GMRES). The solver is preconditioned by a sparse direct solver (SuperLU) applied to the Jacobian matrix, which is computed at the beginning of the step and lagged for subsequent inner iterations. The continuation step size $\delta\sigma$ adapts dynamically, starting from $\delta\sigma=0.2$ at $\sigma=0$ and increasing to $\delta\sigma=2$ for $\sigma\geq 20$.

The continuation proceeds smoothly and the solution converges within 2--4 Newton iterations at every step, achieving a residual $L^\infty$ norm below $10^{-10}$ as evaluated at the collocation points of the spline. See Figure \ref{fig:evolution}.
\smallskip

{\bf{(2) Stability analysis.}} To discover instability, for any $\sigma\geq 80$ we seek generalized eigenvalues $\lambda$ and eigenvectors $(\widetilde{\Psi}_2, \widetilde{\Omega}_2)$ of the linearized system of equations \eqref{Blt12} with $0$ boundary data on the far field boundaries $\{v=A\}$ and $\{u=A\}$. 
The perturbation functions $\wPsi_2, \wOmega_2$ are expanded in the same tensor product basis of quartic B-splines used for the base state.

\begin{figure}[ht]
\centering
\includegraphics[width=0.8\textwidth]{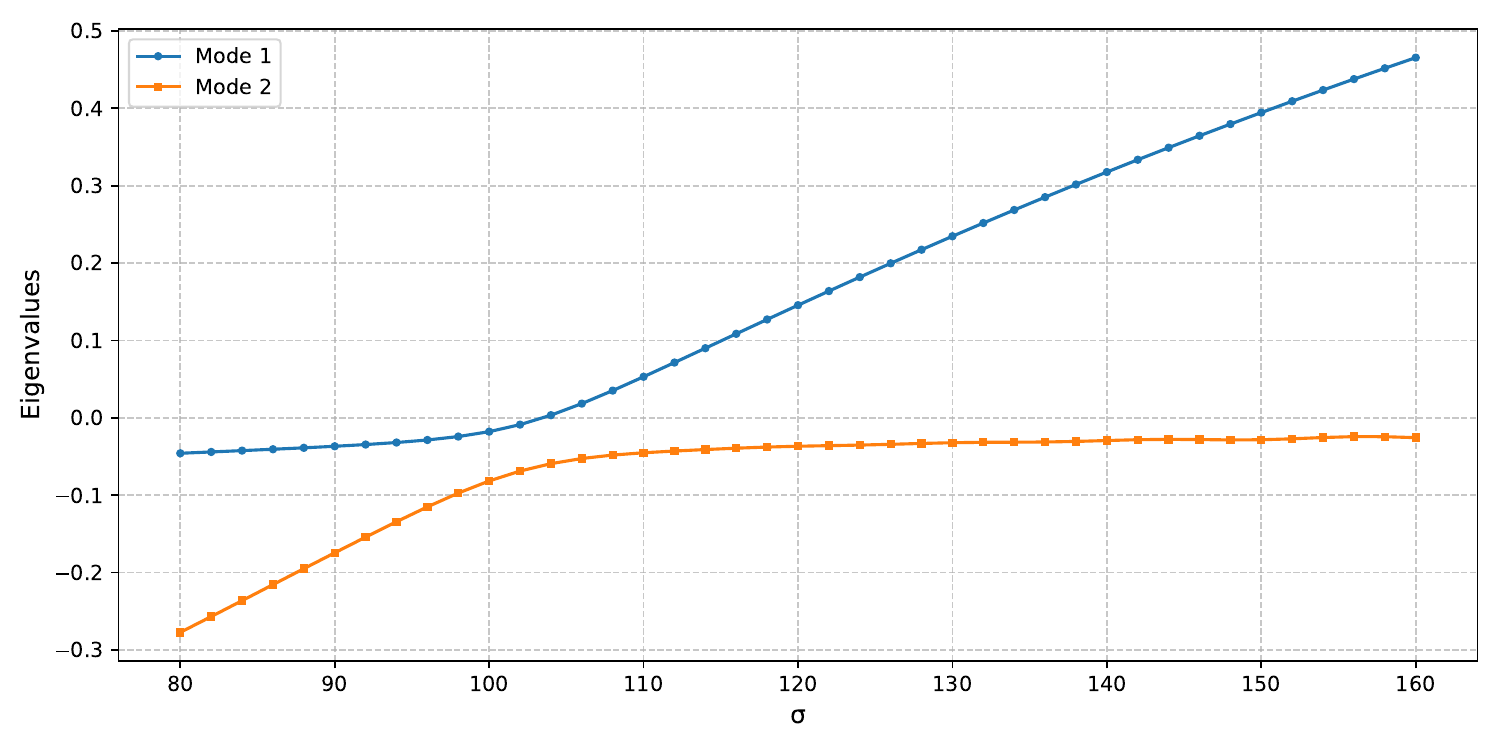}
\caption{The two highest eigenvalues detected by the solver during the $\sigma$-continuation. Both curves represent real single eigenvalues.}
\label{fig:2eig}
\end{figure}

The generalized eigenvalue problem is solved using the Implicitly Restarted Arnoldi Method (IRAM) via the ARPACK library, employing a shift-invert spectral transformation. At each continuation step, the solver attempts to locate eigenvalues using shifts predicted from the leading eigenvalue of the previous step. We strictly filter the computed spectra to retain only non-oscillatory modes with $|\Im\lambda|\ll 1$ that satisfy a residual tolerance of less than $10^{-10}$.

There are two types of eigenfunctions of the original linearized system \eqref{Blt2}. The important ones are those that are even in $z$, which we generate by solving the system \eqref{Blt12} in the square $[0,A]^2$.
This is the case that leads to unstable modes, since the leading eigenvalue $\lambda_\ast$ of the system \eqref{Blt12} crosses the imaginary axis during the $\sigma$-continuation. Indeed, we observe numerically that one single real eigenvalue becomes positive (unstable) at  $\sigma\approx 103$, and grows steadily as $\sigma$ increases towards $160$. The eigenvalue solver detects one other real eigenvalue that increases towards $0$, but never crosses the axis before $\sigma$ reaches $160$. See Figure \ref{fig:2eig}.

\subsection{Step 3: refinement}
In the next stage, we fix $\sigma=130$, a value for which we found numerically that the base solution is linearly unstable, and seek high-precision numerical solutions $(\Psi_2,\Omega_2)$ of the nonlinear problem \eqref{var17} and $(\lambda_\ast, \widetilde{\Psi}_2, \widetilde{\Omega}_2)$ of the eigenvalue problem \eqref{Blt12}.

We initialize the computation using the solution $(\Psi_2,\Omega_2)$ obtained from the continuation phase. This predictor is interpolated onto a high-order tensor product B-spline basis of degree 10. To resolve the fine spatial structures of the flow, we employ a non-uniform knot distribution defined by a multi-zone density function. The grid density is highest in the near-field region $u\in[0,10], v\in[0,4]$ and transitions smoothly to a coarser resolution in the far field, ensuring spectral-like accuracy where gradients are steepest. See Figure \ref{fig:knots}.

\begin{figure}[!t]
\centering
\includegraphics[width=0.8\textwidth]{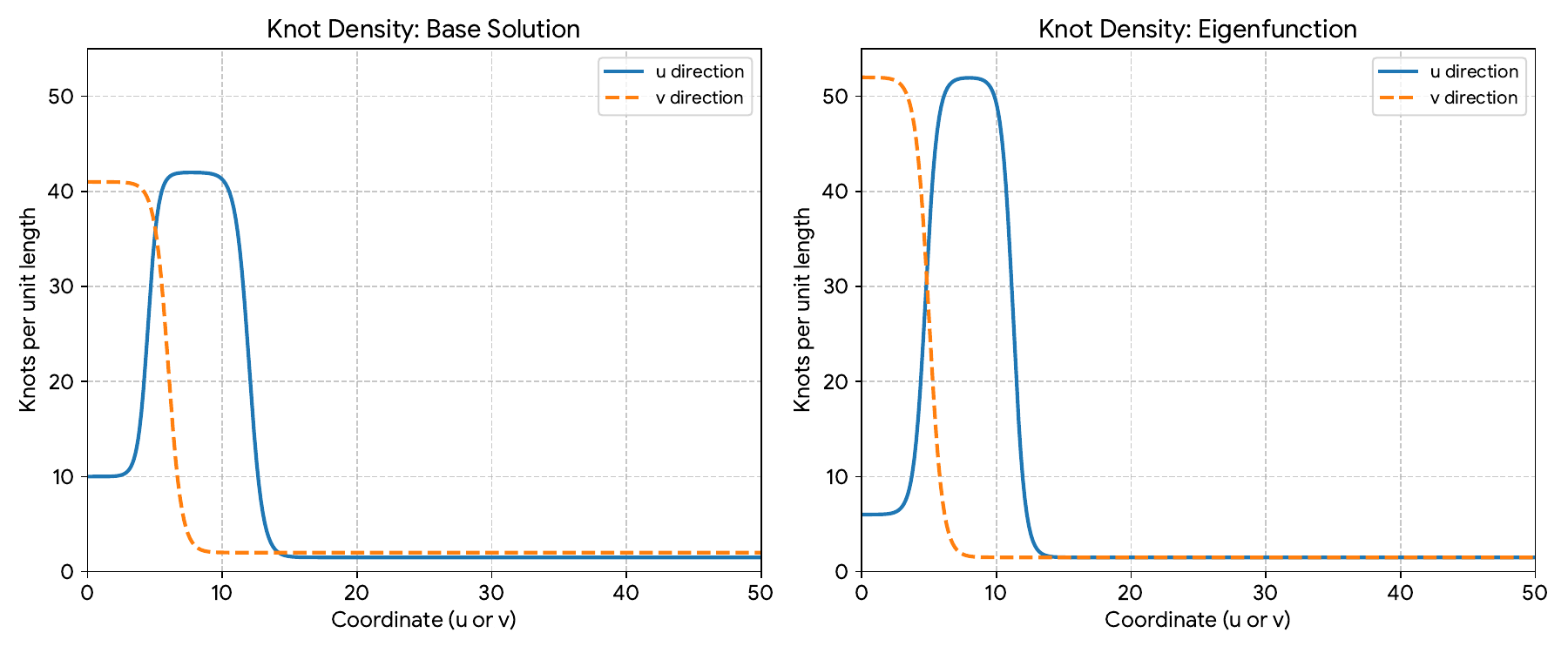}
\caption{Knot density (knots per unit of $u$ or $v$) of the splines used in the refinement stage, showing an adaptive multi-zone grid of degree-10 B-splines, concentrating up to $\sim 50$ knots per unit in the regions where the solution gradients are steepest, and tapering to $\sim 2$ knots per unit in the far field.}
\label{fig:knots}
\end{figure}

We then solve the system \eqref{var17}, subject to improved asymptotic boundary conditions on the far-field boundaries. The improved boundary conditions are obtained from a more precise formal expansion of the solution at infinity in the physical $(r,z)$ space (see subsection \ref{numexp} for details). In the $(u,v)$ computational domain these conditions are $(\Psi_2,\Omega_2)=(\Psi_2^a,\Omega_2^a)$ on the far-field boundaries $\{v=A\}$ and $\{u=A\}$, where
\begin{equation}\label{var13ref}
\begin{split}
&\Psi^a_2(u,v)=\frac{\sigma}{U_0(u)+U_0(v)}+\frac{\sigma^2(3U_0(u)-7U_0(v))}{3(U_0(u)+U_0(v))^3}-\frac{6\sigma }{(U_0(u)+U_0(v))^2},\\
&\Omega^a_2(u,v)=\frac{-6\sigma}{(U_0(u)+U_0(v))^2}+\frac{12\sigma^2(3U_0(v)-2U_0(u))+24\sigma (U_0(u)+U_0(v))}{(U_0(u)+U_0(v))^4}.
\end{split}
\end{equation}
It is important to use these improved asymptotics at this stage, instead of the more basic asymptotics \eqref{var13}, in order to get better matching of the numerical solution and the asymptotic data, thus reducing significantly the gluing error of the global solutions. 

The nonlinear system is solved using a modified Newton-Raphson method: we construct the exact sparse Jacobian matrix and compute its LU factorization using a direct sparse solver (SuperLU). The result is a high-precision solution of the system \eqref{var17} with a residual norm below $3 \times 10^{-11}$ at the collocation points.

Finally, we perform the stability analysis on this refined state. Using the converged high-precision solution $(\Psi_2,\Omega_2)$, we construct the linearized operator on the same degree-10 spline basis and solve the generalized eigenvalue problem using the Implicitly Restarted Arnoldi Method (IRAM). We target the leading unstable mode using a shift of $\rho=0.23$, as estimated in the first stage by the eigenvalue solver. The solver converges to the unstable eigenvalue 
\begin{equation}\label{eigennum}
\lambda_\ast=0.235059597921
\end{equation}
and the computed eigenpair satisfies the discretized equations to a high degree of accuracy, with a residual norm of less than $3\times 10^{-12}$ at the collocation points. The solution is validated using dense residuals: the true eigenpair has a dense residual $\approx 10^{-10}$ while all the other eigenpairs returned by the solver are spurious and have dense residuals $\approx 10^{-2}$.

\subsection{Step 4: conclusions} As a result of this procedure we produce approximate numerical solutions $(\Psi_2,\Omega_2)$ and $(\lambda_\ast,\wPsi_2,\wOmega_2)$ of the systems \eqref{var17} and \eqref{Blt12} respectively, on the computational domain $(u,v)\in[0,50]^2$. The main point is that the residuals of these solutions  are uniformly small in the square $[0,50]^2$ (not just at the collocation points),
\begin{equation}\label{resid1}
\begin{split}
&\|\mathrm{Res}(\Psi_2)\|_{L^\infty([0,50]^2)}\leq 9\times 10^{-12},\qquad \|\mathrm{Res}(\Omega_2)\|_{L^\infty([0,50]^2)}\leq 5\times 10^{-10},\\
&\|\mathrm{Res}(\wPsi_2)\|_{L^\infty([0,50]^2)}\leq 8\times 10^{-11},\qquad \|\mathrm{Res}(\wOmega_2)\|_{L^\infty([0,50]^2)}\leq 3\times 10^{-10},\\
\end{split}
\end{equation}
where the residuals denote the expressions in the left-hand sides of \eqref{var17} and \eqref{Blt12}, and the $L^\infty$ norms are estimated by evaluating these residuals on a uniform dense $6000\times 6000$ grid in the $(u,v)$ space.

Moreover, we can verify numerically that the solutions $(\Psi_2,\Omega_2)$ are well matched to their asymptotic states $\Psi_2^a$ and $\Omega_2^a$ defined in \eqref{var13ref}, i.e.
\begin{equation}\label{resid2}
\begin{split}
&\|\Psi_2-\Psi_2^a\|_{L^{\infty}([0,50]^2\setminus [0,40]^2)}\leq 1.2\times 10^{-8}, \quad \|\Omega_2-\Omega_2^a\|_{L^{\infty}([0,50]^2\setminus [0,40]^2)}\leq 1.4\times 10^{-13},\\
&\|\wPsi_2\|_{L^{\infty}([0,50]^2\setminus [0,40]^2)}\leq 3.2\times 10^{-8}, \qquad\quad\,\,\,\|\wOmega_2\|_{L^{\infty}([0,50]^2\setminus [0,40]^2)}\leq 3.5\times 10^{-13}.
\end{split}
\end{equation}

\begin{figure}[ht]
\centering
\includegraphics[width=\textwidth]{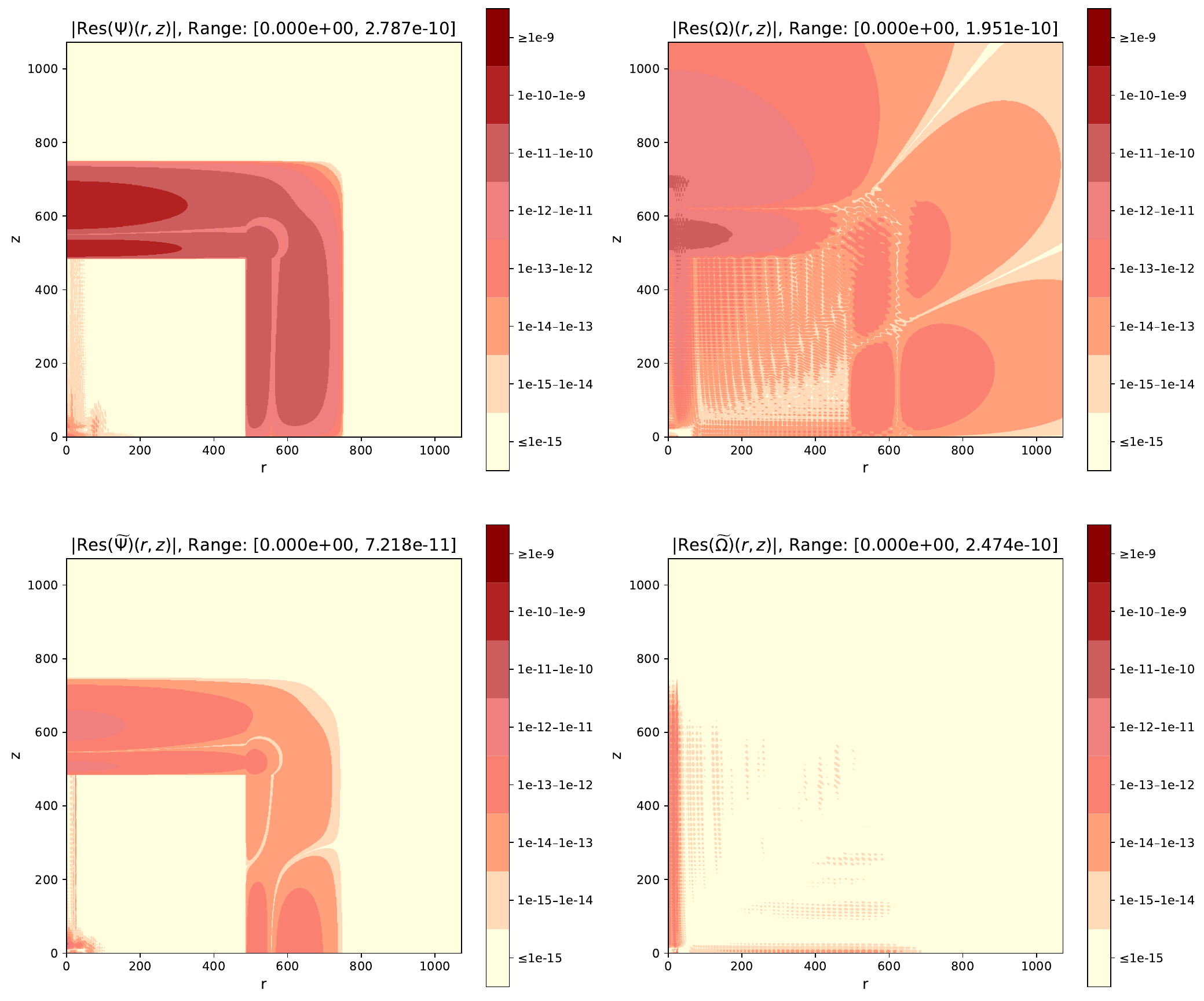}
\caption{Residual heatmaps of the solutions $\Psi$, $\Omega$, $\widetilde{\Psi}$, $\widetilde{\Omega}$ in the physical space $\R^2$ (compare with the main bounds \eqref{MainNum2} and \eqref{asympb}). Notice that the residuals of the stream functions $\Psi, \widetilde{\Psi}$ are more significant in the gluing region, while the residuals of the vorticities $\Omega, \widetilde{\Omega}$ are largest near the point $(r,z)=(26,0)$. The residuals in the gluing region can be further reduced by improving numerically the ansatz in the asymptotic region (as explained at the end of subsection \ref{numexp}), but we do not pursue such improvements in this paper.}
\label{fig:global_res}
\end{figure}

We can thus define the global solutions $\Psi_2^\ast, \Omega_2^\ast, \wPsi_2^\ast, \wOmega_2^\ast:[0,\infty)^2\to\R$ by the gluing formulas 
\begin{equation}\label{glue}
\begin{split}
&\Psi^\ast_2(u,v):=[1-\phi(u)\phi(v)]\Psi^a_2(u,v)+\phi(u)\phi(v)\Psi_2(u,v),\\
&\Omega^\ast_2(u,v):=[1-\phi(u)\phi(v)]\Omega^a_2(u,v)+\phi(u)\phi(v)\Omega_2(u,v),\\
&\wPsi^\ast_2(u,v):=\phi(u)\phi(v)\wPsi_2(u,v),\qquad\wOmega^\ast_2(u,v):=\phi(u)\phi(v)\wOmega_2(u,v),
\end{split}
\end{equation}
where $\phi:\R\to[0,1]$ is a cutoff function equal to $1$ in the interval $[0,A-\rho]$ and equal to $0$ in the interval $[A,\infty)$. In our case $A=50$, $\rho=10$, and we define
\begin{equation}\label{gluephi}
\begin{split}
&\phi(x):=\phi_0((x-A)/\rho),\qquad x\in[A-\rho, A].\\
&\phi_0(y):=1-30\int_{-1}^y z^2(z+1)^2\,dz=-6y^5-15y^4-10y^3,\qquad y\in[-1,0].
\end{split}
\end{equation}
The residuals in the gluing region $\mathcal{R}:=[0,A]^2\setminus[0,A-\rho]^2$ can be estimated numerically by
\begin{equation}\label{resid6}
\begin{split}
&\|\mathrm{Res}(\Psi^\ast_2)\|_{L^\infty(\mathcal{R})}\leq 6\times 10^{-13},\qquad \|\mathrm{Res}(\Omega^\ast_2)\|_{L^\infty(\mathcal{R})}\leq 7\times 10^{-14},\\
&\|\mathrm{Res}(\wPsi^\ast_2)\|_{L^\infty(\mathcal{R})}\leq 2\times 10^{-12},\qquad \|\mathrm{Res}(\wOmega^\ast_2)\|_{L^\infty(\mathcal{R})}\leq 4\times 10^{-11}.
\end{split}
\end{equation}

Finally, we pass to the physical space and define the global solutions $\Psi,\Omega,\wPsi, \wOmega:\R^2\to\R$ by the formulas $(\Psi(r,z),\Omega(r,z))=(z\Psi_2^\ast(u,v),z\Omega_2^\ast(u,v))$, $(\wPsi(r,z),\wOmega(r,z))=(\wPsi_2^\ast(u,v),\wOmega_2^\ast(u,v))$; see \eqref{var6}. This leads to the solutions described in our Main Numerical Result \ref{MainNumRes}. The conclusions stated in part (ii) (the parity properties of the solutions and their asymptotic behavior) follow (rigorously) directly from the definitions. The main numerical conclusions in part (i) are the residual bounds \eqref{MainNum2}. To verify them we notice that 
\begin{equation}\label{resid7}
\begin{split}
&\mathrm{Res}(\Psi)(r,z)=z\mathrm{Res}(\Psi_2^\ast)(u,v),\qquad \mathrm{Res}(\Omega)(r,z)=z\mathrm{Res}(\Omega_2^\ast)(u,v),\\
&\mathrm{Res}(\wPsi)(r,z)=\mathrm{Res}(\wPsi_2^\ast)(u,v),\qquad\,\,\,\mathrm{Res}(\wOmega)(r,z)=\mathrm{Res}(\wOmega_2^\ast)(u,v),
\end{split}
\end{equation}
as a consequence of the definitions. We can use these formulas to estimate the residuals of the functions $\Psi, \Omega, \wPsi, \wOmega$ in the interior domain, using the residuals of the functions $\Psi_2^\ast, \Omega_2^\ast, \wPsi_2^\ast, \wOmega_2^\ast$ estimated earlier in the computational domain. The bounds in the exterior domain are better, as described in part (ii).

\section{Proof of Theorem \ref{MainThm}}\label{thimpl}

In this section we recall some results from \cite{JiSv1, JiSv2} and provide the proof of Theorem \ref{MainThm}. Define the space of divergence free vector fields 
\begin{equation}\label{pot0}
  \begin{split}
  &X:=\big\{\phi\in L^2\cap L^4(\R^3):\,{\rm div}\,\phi=0\big\},\qquad \mathcal{D}:= \big\{\phi\in X:\,x\cdot\nabla\phi\in X,\,\, \Delta\phi \in X\big\},\\
  &Y:=\Big\{U\in C^2(\R^3):\,{\rm div}\,U=0, \quad \sup_{x\in\R^3, \,|\alpha|=0,1,2}\langle x\rangle^{|\alpha|+1} |\nabla^{\alpha}U(x)|<\infty\Big\},
  \end{split}
\end{equation}
with the corresponding norms 
\begin{equation}\label{pot0.1}
\begin{split}
    &\|\phi\|_X:=\|\phi\|_{L^2\cap L^4(\R^3)}, \quad \|\phi\|_\mathcal{D}:=\|\phi\|_X + \|\nabla^2\phi\|_X + \|x\cdot\nabla\phi\|_X, \\
    &\|U\|_Y:= \sup_{x\in\R^3, \,|\alpha|\leq 2}\Big[\langle x\rangle^{|\alpha|+1}|\nabla^{\alpha}U(x)|\Big]. 
\end{split}    
\end{equation}
  For $U\in Y$, we introduce the linearized operator $\mathcal{L}_{U}:\mathcal{D}\subseteq X\to X$ as follows
  \begin{equation}\label{pot1}
      \mathcal{L}_{U}\phi:=\Delta \phi+\frac{1}{2}\mathbf{x}\cdot\nabla \phi+\frac{1}{2}\phi-\Pi(U\cdot\nabla \phi+\phi\cdot\nabla U),
  \end{equation}
  for any $\phi\in\mathcal{D}$, where $\Pi$ denotes the Leray projection. For any $\phi\in\mathcal{D}$, we also define $\mathcal{L}_0, K_U: \mathcal{D}\subseteq X\to X$ by
  $$\mathcal{L}_0\phi:=\Delta \phi+\frac{1}{2}\mathbf{x}\cdot\nabla \phi+\frac{1}{2}\phi, \qquad K_U\phi:=-\Pi(U\cdot\nabla \phi+\phi\cdot\nabla U).$$

The operator $\mathcal{L}_U$ generates a continuous semigroup $e^{t\mathcal{L}_U}$ on $X$ for $t\ge0$, which can be seen as follows. Define for $x\in\R^3,\,\,t\ge0$,
  \begin{equation}\label{pot1.01}
      h(x,t):=\frac{1}{\sqrt{t+1}}\phi\Big(\frac{x}{\sqrt{t+1}}, \log(t+1)\Big), \quad v(x,t):=\frac{1}{\sqrt{t+1}}U\Big(\frac{x}{\sqrt{t+1}},\log(t+1)\Big).
  \end{equation}
  Then $\phi(s)\in \mathcal{D}\subseteq X$ solving
  \begin{equation}\label{pot1.02}
      \partial_s\phi=\mathcal{L}_U\phi
  \end{equation}
  for $s\geq 0$ with initial data $\phi_0\in X$, is equivalent to $h$ solving the linear Stokes equation (with lower order terms)
  \begin{equation}\label{pot1.001}
      \partial_th-\Delta h + \Pi(v\cdot\nabla h+ h\cdot\nabla v)=0,\quad {\rm div\,}h=0,
  \end{equation}
  for $t\geq 0$ with the same initial data. 

  Similarly, there is a correspondence on the level of resolvents. Consider the resolvent equation for $\lambda:=\alpha+i\beta\in\C$ with $\alpha, \beta\in\R$ and $\alpha>-1/4$,
  \begin{equation}\label{pot2}
      \Delta \phi+\frac{1}{2}\mathbf{x}\cdot\nabla \phi+\frac{1}{2}\phi -\Pi(U\cdot\nabla \phi+\phi\cdot\nabla U)-\lambda\phi=F,
  \end{equation}
  for some $F\in X$. Set
  \begin{equation}\label{pot3}
      u(x,t):=t^{-\frac{1}{2}+\alpha+i\beta }\phi\Big(\frac{x}{\sqrt{t}}\Big),\quad v(x,t)=t^{-\frac{1}{2}}U\Big(\frac{x}{\sqrt{t}}\Big),\qquad f(x,t):=t^{-\frac{3}{2}+\alpha+i\beta}F\Big(\frac{x}{\sqrt{t}}\Big). 
  \end{equation}
  It follows from \eqref{pot2} that
  \begin{equation}\label{pot4}
      \partial_tu-\Delta u + \Pi(v\cdot\nabla u + u\cdot\nabla v) = -f(x,t),\quad {\rm div\,}u=0,
  \end{equation}
  with $u|_{t=0}\equiv0$. 

The main properties we need for $\mathcal{L}_U$ and $e^{t\mathcal{L}_U}$ are summarized in the following proposition.

\begin{proposition}\label{pot0.3}
    Let $U\in Y$, $\mathcal{L}_U$, $K_U$ and $\mathcal{L}_0$ be defined as above. 
    Then we have the following conclusions:

  (i) The spectrum of $\mathcal{L}_{0}$ is contained in $\Sigma:=\{\lambda\in\C:\,\Re \lambda\leq -1/4\}$ and we have the bounds for $\lambda=\alpha+i\beta\not\in\Sigma$ and all $\phi=(\lambda-\mathcal{L}_0)^{-1}g$,
  \begin{equation}\label{port1.11}
      \|\phi\|_{X}\lesssim \frac{1}{\alpha + 1/4}\|g\|_X, \qquad\|\nabla^2\phi\|_X+\|(-\frac{1}{2}x\cdot\nabla+\alpha+i\beta)\phi\|_X\lesssim \frac{\alpha+1}{\alpha + 1/4}\|g\|_X. 
  \end{equation}
Moreover, the associated semigroup generated by $\mathcal{L}_0$ satisfies the bound for $t\ge0$ and $g\in X$,
\begin{equation}\label{port1.2}
    \|e^{\mathcal{L}_0t}g\|_{X}+\min\{t^{\frac{1}{2}},1\} \|\nabla e^{\mathcal{L}_0t}g\|_{X}+\min\{t,1\}\|\nabla^2e^{\mathcal{L}_0t}g\|_{X}\lesssim e^{-t/4}\|g\|_X. 
\end{equation}

(ii) The operator $K_U$ is relatively compact with respect to $\mathcal{L}_0$, in the sense that $K_U(\lambda-\mathcal{L}_0)^{-1}: X\to X$ is compact for any $\lambda\in\C\setminus\Sigma$. The spectrum $\sigma(\mathcal{L}_U)=\Sigma\cup\{\lambda_j\}$, where the discrete sequence $\lambda_j$ may only accumulate in $\Sigma$. Moreover, for any $\alpha_0>-1/4$, there are only finitely many $\lambda_j$ with $\Re\lambda_j\ge\alpha_0$. 

(iii) Take $m>\max\{-1/4, \Re\lambda_j\}$. The semigroup $e^{t\mathcal{L}_U}$ is a relatively compact perturbation of $e^{t\mathcal{L}_0}$ for any $t\ge0$, and we have the bounds 
\begin{equation}\label{port1.20}
    \|e^{t\mathcal{L}_U}g\|_{X}+\min\{t^{\frac{1}{2}},1\} \|\nabla e^{t\mathcal{L}_U}g\|_{X}+\min\{t,1\}\|\nabla^2e^{t\mathcal{L}_U}g\|_{X}\lesssim_{\|U\|_Y, m} e^{mt}\|g\|_X. 
\end{equation}
\end{proposition}

\begin{proof} Most of the claims were proved in \cite{JiSv2} using  \eqref{pot1.01} and \eqref{pot3} which transform results in the proposition to corresponding results for the linear Stokes equations (with lower order terms). 

(i) The claims follow from Lemma 2.1 and Lemma 2.3 in \cite{JiSv2}. The bounds \eqref{port1.11} are somewhat more quantitative in $\alpha$ but follow from the same arguments. 

(ii) The relative compactness of $K_U$ with respect to $\mathcal{L}_0$ follows from Lemma 2.2 in \cite{JiSv2}. The claim concerning the spectrum structure of $\sigma(\mathcal{L}_U)$ then follows from standard analytic Fredholm theory (or Weyl's theorem for essential spectra).

The analysis in \cite{JiSv2}, while robust and sufficient for our applications, does not rule out the possibility that $\Im\lambda_j$ may be arbitrarily large. We can eliminate this possibility using the following bounds. 
  For any $\lambda=\alpha+i\beta$ with $\alpha>-1/4$, $U\in Y$, $g\in \mathcal{D}\subseteq X$, we show that
  \begin{equation}\label{pot10}
  \begin{split}
   &   \|(\lambda-\mathcal{L}_0)^{-1}K_Ug\|_X \lesssim\frac{1}{(1+|\alpha|+|\beta|)^{1/2}}\|U\|_Y\|g\|_{X}. 
   \end{split}   
  \end{equation}
 This bound implies that $\mathcal{L}_U$ has no eigenvalues for sufficiently large $\Im\lambda$, using the well known resolvent identity
  \begin{equation*}
 (\lambda-\mathcal{L}_0-K_U)^{-1}=(I-(\lambda-\mathcal{L}_0)^{-1}K_U)^{-1}(\lambda-\mathcal{L}_0)^{-1}.
 \end{equation*}

To prove \eqref{pot10} let $g\in X$ and $U\in Y$, and denote 
 the matrix
\begin{equation}\label{decay2.1}
    h:=U\otimes g+g\otimes U, \qquad h_{ij}=U_ig_j+g_iU_j.
\end{equation}
Notice that 
\begin{equation}\label{decay1}
    K_Ug=-\Pi\,{\rm div}(U\otimes g+g\otimes U). 
\end{equation}
We have the bound 
\begin{equation}\label{decay2}
    \|\langle x\rangle U\otimes g\|_{X}+\|\langle x\rangle g\otimes U\|_X\lesssim \|U\|_Y\|g\|_X. 
\end{equation}

We show now that
\begin{equation}\label{decay3}
   \| \langle x\rangle(\lambda-\mathcal{L}_0)^{-1}K_Ug\|_X\lesssim \langle\alpha\rangle^{-1/3}\|g\|_X. 
\end{equation}
In view of the formulae \eqref{pot2}-\eqref{pot4} (with the lower terms being identically zero for the resolvent $(\lambda-\mathcal{L}_0)^{-1}$), it suffices to prove that
\begin{equation}\label{decay4}
\begin{split}
   & \Big\|\langle x\rangle\int_0^1e^{\Delta(1-s)}s^{-\frac{3}{2}+\alpha+i\beta}\Pi\big({\rm div}\,h\big)(\frac{\cdot}{\sqrt{s}})\,ds\Big\|_X\\
   &= \Big\|\langle x\rangle\int_0^1 s^{-1+\alpha+i\beta}e^{\Delta(1-s)}{\rm div}\,\Pi \,h(\frac{\cdot}{\sqrt{s}})\,ds\Big\|_X\lesssim \langle\alpha\rangle^{-1/3}\|\langle x\rangle h\|_X,
\end{split}    
\end{equation}
using standard  estimates on the heat kernel. The desired bounds \eqref{decay3} follow.

To proceed, we rewrite $\phi:=(\lambda-\mathcal{L}_0)^{-1}K_Ug$ as 
\begin{equation*}
      \Delta \phi-\phi-\alpha\phi-i\beta\phi =\Pi \,{\rm div}\,h - \frac{1}{2}{\rm div}\,(\phi\otimes x),
\end{equation*}
thus
\begin{equation*}
    \phi=(\Delta -1-\alpha-i\beta)^{-1}\,{\rm div}\big[\Pi h - \frac{1}{2}(\phi\otimes x)\big]. 
\end{equation*}
The desired bounds \eqref{pot10} then follow from \eqref{decay2}-\eqref{decay3}, since the operator $(\Delta-1-\alpha-i\beta)^{-1}\,{\rm div}$ has an operator norm bounded by $C(1+|\alpha|+|\beta|)^{-1/2}$.

(iii)  The difference $e^{t\mathcal{L}_U} - e^{t\mathcal{L}_0}$ is a compact operator on $X$ for any $t \ge 0$, due to Lemma 2.7 in \cite{JiSv2}. Consequently, by Weyl's theorem, the essential spectrum of the full semigroup is identical to that of the base semigroup, meaning $\sigma_{ess}(e^{t\mathcal{L}_U}) \subseteq \{\lambda \in \mathbb{C} : |\lambda| \le e^{-t/4}\}$. Since there are strictly finitely many discrete eigenvalues outside this stable region (due to (ii)), the spectral radius of $e^{t\mathcal{L}_U}$ is bounded by $e^{mt}$ for any $m>\max\{-1/4,\Re\lambda_j\}$. The desired bounds \eqref{port1.20} follow from Lemma 2.9 and equation (2.21) in \cite{JiSv2}.
\end{proof}

We turn now to the proof of Theorem \ref{MainThm}, which we divide into several steps:

{\bf{Step 1.}} Under the assumptions of Theorem \ref{MainThm} and standard elliptic bootstrapping arguments (similar to the arguments in \cite[Section 4]{JiSv1}), the vector field $U\in C^2(\R^3)$ satisfies
\begin{equation}\label{pot15}
    \Delta U+\frac{1}{2}\mathbf{x}\cdot\nabla U+\frac{1}{2}U-\Pi(U\cdot\nabla U)=0, \quad {\rm div}\,U=0,
\end{equation}
in $\R^3$, with 
\begin{equation}\label{pot15.001}
|\nabla^\alpha(U-e^{\Delta}U_0)(x)|\lesssim\langle x\rangle^{-3-|\alpha|}\qquad\text{ for any }x\in\R^3,\,|\alpha|\leq 2.
\end{equation}
Moreover, the linearized operator $\mathcal{L}_U:\mathcal{D}\subseteq X\to X$ has at least one unstable eigenvalue $\lambda$ with $\Re\lambda>0$ and the corresponding eigenfunction $\Phi\in \mathcal{D}$ satisfying
\begin{equation}\label{pot14}
    |\nabla^{\alpha}\Phi(x)|\lesssim \langle x\rangle ^{-3-|\alpha|}\qquad\text{ for any }x\in\R^3,\,|\alpha|\leq 2.
\end{equation}

Consider the time-dependent version of the equation \eqref{pot15} given by
\begin{equation}\label{pot15.1}
   \partial_sV= \Delta V+\frac{1}{2}\mathbf{x}\cdot\nabla V+\frac{1}{2}V-\Pi\big(V\cdot\nabla V\big), \quad {\rm div}\,V=0,
\end{equation}
with the boundary condition \eqref{pot15.001}. The solution $U$ to \eqref{pot15} can be viewed as a steady state of \eqref{pot15.1}. Since the linearized operator $\mathcal{L}_U$ has unstable eigenvalues, we can construct an unstable manifold near $U$. Some care is needed, however, since $U$ decays like $1/|x|$, so $U\not\in L^2(\R^3)$. The localization argument we provide here is a modification of the argument in \cite{JiSv2}, which does not need the assumption that the maximum real part of all the unstable eigenvalues of $\mathcal{L}_U$ is less than $1/8$. A similar argument was used recently in \cite{HoWaYa}. 

We decompose first
\begin{equation}\label{pot16}
    U_0=U_{0l}+ U_{ 0f},
\end{equation}
where the divergence free field $U_{0l}\equiv U_0$ for $|x|\leq 2R$, ${\rm supp\,} U_{0l}\subseteq B_{4R}$, $U_{0l}\in C^\infty(B_{4R}\backslash \{0\})$, for some $R\gg1$ that we can choose freely to make $U_{0f}$ small. In addition, $U_{0f}$ satisfies the property that for all $x\in\R^3$ and multi-index $\alpha$,
\begin{equation}\label{pot17}
    U_{0f}\in C^\infty(\R^3\backslash B_{2R}), \quad |\nabla^{\alpha}U_{0f}(x)|\lesssim |x|^{-1-|\alpha|}.
\end{equation}
To achieve the decomposition \eqref{pot16} in the divergence-free class, we define the vector potential $A:=U_0\times x$, which is homogeneous of degree $0$, and notice that $U_0=\nabla\times A$ (using the assumptions that $U_0$ is divergence-free and homogeneous of degree $-1$). Let $\chi \in C^\infty_0(B_{4R})$ be a smooth radial cutoff function such that $\chi \equiv 1$ on $B_{2R}$. We define the localized velocity fields by
\begin{equation}\label{pot_potential}
U_{0l} := \nabla \times (\chi A),\qquad U_{0f} := \nabla \times ((1-\chi) A),
\end{equation}
and the desired decomposition follows.

We are now looking for solutions of the Navier--Stokes equations \eqref{Nav1} of the form
\begin{equation}\label{pot20}
    \widetilde{u}(x,t)=\frac{1}{\sqrt{t}}U(\frac{x}{\sqrt{t}})-e^{\Delta t}U_{0f}(x)+h(x,t):=u_1(x,t)-U_f(x,t)+h(x,t),
\end{equation}
where $h$ satisfies on $\R^3\times(0,1]$ the perturbed Navier--Stokes equation
\begin{equation}\label{pot21}
\begin{split}
 &\partial_th-\Delta h +\Pi\big(u_1\cdot\nabla h+h\cdot\nabla u_1-U_f\cdot\nabla h-h\cdot\nabla U_f+h\cdot\nabla h\big) =\,g(x,t), \\
 &{\rm div}\, h=\,0,
\end{split}    
\end{equation}
with zero initial data and source term given by
\begin{equation}\label{pot22}
    g:=\Pi\big(u_1\cdot\nabla U_f+U_f\cdot\nabla u_1-U_f\cdot\nabla U_f\big).
\end{equation}

To solve \eqref{pot21}, we make the transformation for $x\in\R^3, t>0$,
\begin{equation}\label{pot23}
    U_f(x,t):=\frac{1}{\sqrt{t}}b\big(\frac{x}{\sqrt{t}},\log t\big),\quad h(x,t):=\frac{1}{\sqrt{t}}H\big(\frac{x}{\sqrt{t}},\log t\big), \quad g(x,t):= \frac{1}{t^{3/2}}G\big(\frac{x}{\sqrt{t}},\log t\big).
\end{equation}
It follows from \eqref{pot20} and \eqref{pot22} that for $s\leq0$,
\begin{equation}\label{pot23.1}
    G=\Pi\big[U\cdot\nabla b +b \cdot\nabla U- b\cdot\nabla b\big]
\end{equation}
and that $H$ satisfies the equation for $s\leq0$, 
 \begin{equation}\label{pot24}
 \begin{split}
    & \partial_s H=\mathcal{L}_UH+\Pi\big[b\cdot\nabla H+H\cdot\nabla b+H\cdot\nabla H\big] +G, \\
    &{\rm div\,}H=0. 
 \end{split}    
 \end{equation}

{\bf{Step 2.}} Thanks to the spectral properties established in Proposition \ref{pot0.3}, we can directly construct a family of solutions bifurcating from $U$. By Proposition \ref{pot0.3}, the spectrum of $\mathcal{L}_U$ in the right half-plane consists of strictly finitely many discrete eigenvalues. We define the finite-rank unstable projection $P_u$ via the Riesz integral 
\begin{equation}\label{def_Pu}
P_u = \frac{1}{2\pi i} \oint_{\Gamma_u} (z I - \mathcal{L}_U)^{-1} \, dz,
\end{equation}
where $\Gamma_u$ is a smooth closed contour in the right half-plane strictly enclosing $\sigma_u$, and let $P_{cs} = I - P_u$ denote the complementary center-stable projection.

By the spectral mapping theorem and the bounds established in \eqref{port1.20}, the restricted semigroups satisfy the bounds
\begin{equation}\label{semigroup_bounds}
\begin{split}
&\|e^{s\mathcal{L}_U} P_u g\|_X + \|\nabla e^{s\mathcal{L}_U} P_u g\|_X \lesssim e^{\gamma_u s/2} \|g\|_X \quad \text{for } s \le 0, \\
&\|e^{s\mathcal{L}_U} P_{cs} g\|_X + \min(1, s^{1/2}) \|\nabla e^{s\mathcal{L}_U} P_{cs} g\|_X \lesssim_\epsilon e^{\epsilon s} \|g\|_X \quad \text{for } s > 0,
\end{split}
\end{equation}
where $\gamma_u > 0$ is the minimum real part of the unstable eigenvalues, and $\epsilon > 0$ can be chosen arbitrarily small.

We rewrite equation \eqref{pot24} using the variation of parameters formula. For a given $\phi_{u0} \in \text{Ran}(P_u)$, we seek a trajectory $H(s)$ satisfying the boundary condition $P_u(H)(0)=\phi_{u0}$ and
\begin{equation}\label{unstable_integral}
\begin{split}
&H(s) = e^{s\mathcal{L}_U}\phi_{u0} + \int_{-\infty}^s e^{(s-\tau)\mathcal{L}_U} P_{cs} \mathcal{N}(H)(\tau) \, d\tau - \int_s^0 e^{(s-\tau)\mathcal{L}_U} P_u \mathcal{N}(H)(\tau) \, d\tau,\\
&\mathcal{N}(H):=\Pi\big[b\cdot\nabla H+H\cdot\nabla b+H\cdot\nabla H\big] +G.
\end{split}
\end{equation}
We solve this equation by a fixed point argument in the Banach space $E_\eta$ defined by
\begin{equation}\label{Enorm}
\begin{split}
&E_\eta:=\big\{h\in C((-\infty,0]:X^1):\,\|h\|_{E_\eta}:= \sup_{s \le 0} e^{-\eta s} \|h(s)\|_{X^1}<\infty\big\},\\
&X^1:=\big\{f\in X:\,\|f\|_{X^1}:=\|f\|_X+\|\nabla f\|_{X}<\infty\big\},
\end{split}
\end{equation}
with a decay rate $\eta$ sufficiently small, $\eta:=\min\{1/100,\gamma_u/4\}$. 

The function $b$ defined in \eqref{pot23} satisfies the bounds
\begin{equation}\label{pot25}
    \begin{split}
        &\|b(s)\|_{L^\infty(\R^3)}\lesssim e^{s/2}R^{-1}, \quad \|b(s)\|_{L^4(\R^3)}\lesssim e^{s/8}R^{-1/4},\\
        &\|\nabla b(s)\|_{L^\infty(\R^3)}\lesssim e^sR^{-2},\quad \|\nabla b(s)\|_{L^4(\R^3)}\lesssim e^{5s/8}R^{-5/4}. 
    \end{split}
\end{equation}
It follows from \eqref{pot15.001} and \eqref{pot25} that for some $\epsilon_R\to0+$ as $R\to\infty$,
\begin{equation*}
    \begin{split}
        &\|(b\cdot\nabla b)(s)\|_{L^2}\lesssim e^{3s/4}\epsilon_R,\qquad \|(b\cdot\nabla b)(s)\|_{L^4}\lesssim e^s\epsilon_R,\\
        &\|(b\cdot\nabla U)(s)\|_{L^2}\lesssim e^{s/2}\epsilon_R,\qquad \|(U\cdot\nabla b)(s)\|_{L^2}\lesssim e^{5s/8}\epsilon_R,\\
        &\|(b\cdot\nabla U)(s)\|_{L^4}\lesssim e^{s/2}\epsilon_R,\qquad \|(U\cdot\nabla b)(s)\|_{L^4}\lesssim e^{s}\epsilon_R.
    \end{split}
\end{equation*}
Therefore, using also \eqref{pot23.1},
\begin{equation}\label{pot26.1}
    \|G(s)\|_X\lesssim e^{s/2}\epsilon_R.
\end{equation}
Moreover, the definition \eqref{unstable_integral} and the bounds \eqref{pot25} show that
\begin{equation}\label{pot26.2}
\|\mathcal{N}(H_1)(s)-\mathcal{N}(H_2)(s)\|_X\lesssim \|(H_1-H_2)(s)\|_{X^1}(e^{s/2}\epsilon_R+\|H_1(s)\|_{X^1}+\|H_2(s)\|_{X^1})
\end{equation}
for any $H_1,H_2\in C((-\infty,0]:X^1)$. A standard fixed-point argument using \eqref{semigroup_bounds}, \eqref{pot26.1}, and \eqref{pot26.2} shows that the equation \eqref{unstable_integral} admits a unique solution 
\begin{equation}\label{pot26.3}
H\in C((-\infty,0]:X^1),\qquad \sup_{s\in (-\infty,0]}e^{-\eta s}\|H(s)\|_{X^1}\lesssim \epsilon_R+\|\phi_{u0}\|_{X^1},
\end{equation}
provided that $\|\phi_{u0}\|_{X^1}\ll 1$ and $R$ is chosen sufficiently large.
\medskip

{\bf{Step 3.}} The solution $\widetilde{u}\in C([0,1]:L^2)$ given by
\begin{equation}\label{pot28}
   \widetilde{u}(x,t):= \frac{1}{\sqrt{t}}U(\frac{x}{\sqrt{t}})-e^{t\Delta }U_{0f}(x)+\frac{1}{\sqrt{t}}H(\frac{x}{\sqrt{t}},\log t)
\end{equation}
then solves the Navier--Stokes equation with initial data $U_{0l}$. Comparing the values at $t=1$, we conclude that different choices of $\phi_{u0}\in \text{Ran}(P_u)$ lead to different solutions $\widetilde{u}$ with the same initial data $U_{0l}\in C^\infty(\R^3\backslash\{0\})$, which is compactly supported and has a singularity of the order $O(1/|x|)$ as $x\to0$. 

Clearly, we have $\sqrt{t}|\widetilde{u}(x,t)|\lesssim 1$ and $|x||\widetilde{u}(x,t)|\lesssim 1$ for any $x,t\in\R^3\times (0,1]$, thus $\widetilde{u}\in L^\infty_tL^{3,\infty}_x$ and $\sqrt{t}\|\widetilde{u}(t)\|_{L^\infty}\lesssim 1$. Moreover, we can rewrite 
\begin{equation}\label{pot28.1}
\widetilde{u}(x,t)-e^{t\Delta}U_{0l}(x)= \frac{1}{\sqrt{t}}(U-e^{\Delta}U_0)(\frac{x}{\sqrt{t}})+\frac{1}{\sqrt{t}}H(\frac{x}{\sqrt{t}},\log t),
\end{equation}
so $\widetilde{u}\in C([0,1]:H^\alpha_x)$, $\alpha<1/2$, as a consequence of \eqref{pot15.001} and \eqref{pot26.3}. 

Finally, by parabolic regularization using the identity \eqref{pot24} and the bounds \eqref{port1.20} and \eqref{pot26.1}--\eqref{pot26.2}, we have $\|H(s)\|_{H^{\alpha}}\lesssim_\alpha \epsilon_R+\|\phi_{u0}\|_{X^1}$ for any $s\in(-\infty,0]$ and $\alpha\in[0,2)$. The identity \eqref{pot28.1} then shows that $\widetilde{u}\in L^2_tH^{\alpha+1}_x$ for any $\alpha\in[0,1/2)$, as desired.

The strong energy identity \eqref{pr10} follows for all $t_0\leq t_1\in(0,1]$, since the solutions $\widetilde{u}$ are regular in this interval. It then follows for all $t_0\leq t_1\in[0,1]$, by the continuity of $\widetilde{u}(t)$ in $L^2$.
\medskip

{\bf{Step 4.}} Assume now that $U,\widetilde{U}$ are axi-symmetric swirl-free vector fields, and we would like to prove the same property for the resulting solutions of the Navier--Stokes equation. Notice that a vector field $V=(V_1,V_2,V_3)$ in $\R^3$ is ASSF if and only if 
\begin{equation}\label{pot29.1}
\begin{split}
&O(V_1)=-V_2, \qquad O(V_2)=V_1,\qquad O(V_3)=0,\\
&-x_2V_1(x)+x_1V_2(x)=0\qquad \text{ for all }x\in\R^3,
\end{split}
\end{equation}
where $O:=x_1\partial_2-x_2\partial_1$. 

We define the spaces
\begin{equation}\label{pot29.2}
\begin{split}
 &X_{sym}:=\big\{\phi\in X:\, \phi\text{ is axially symmetric and swirl-free}\big\},\qquad\mathcal{D}_{sym}:=\mathcal{D}\cap X_{sym},\\
 &Y_{sym}:=\big\{U\in Y:\, U\text{ is axially symmetric and swirl-free}\big\}.
 \end{split}
\end{equation}
The characterization \eqref{pot29.1} can be used to show that if $U\in Y_{sym}$ then the operator $\mathcal{L}_U$ preserves the symmetries, i.e. $\mathcal{L}_U:\mathcal{D}_{sym}\to X_{sym}$ and $(\lambda-\mathcal{L}_U)^{-1}:X_{sym}\to X_{sym}$. Moreover, the vector fields $U_{0l}$ and $U_{0f}$ defined in \eqref{pot_potential} are also ASSF. In particular, the vector fields $g$ and $G$ defined in \eqref{pot22} and \eqref{pot23.1} are ASSF. Therefore the vector field $H$ defined in \eqref{unstable_integral} is ASSF, provided that $\phi_{u0}\in X_{sym}$. This completes the proof of Theorem \ref{MainThm}.

\section{Remarks}\label{sec:4} We conclude this article with several remarks.

\subsection{Numerical stability} Similar analysis can be carried out for other values of the parameters $(A,p)$ that describe the computational domain. We have run the continuation procedure for the pairs $(A,p)=(40,4)$, $(A,p)=(50,4)$, $(A,p)=(25,5)$, and $(A,p)=(15,6)$, corresponding to physical domains of sizes $L\approx 485$, $L\approx 751$, $L\approx 995$, and $L\approx 1182$ respectively. The leading eigenvalue exhibits remarkable stability over all these domains. Similar stability is observed by running the $\sigma$-continuation directly in the physical space, using equations \eqref{sf1}--\eqref{Blt2}, which is possible for smaller values of $L$ such as $L=100$, $L=150$, $L=200$, using either B-splines or finite elements with FEniCS.

\begin{figure}[ht]
\centering
\includegraphics[width=0.8\textwidth]{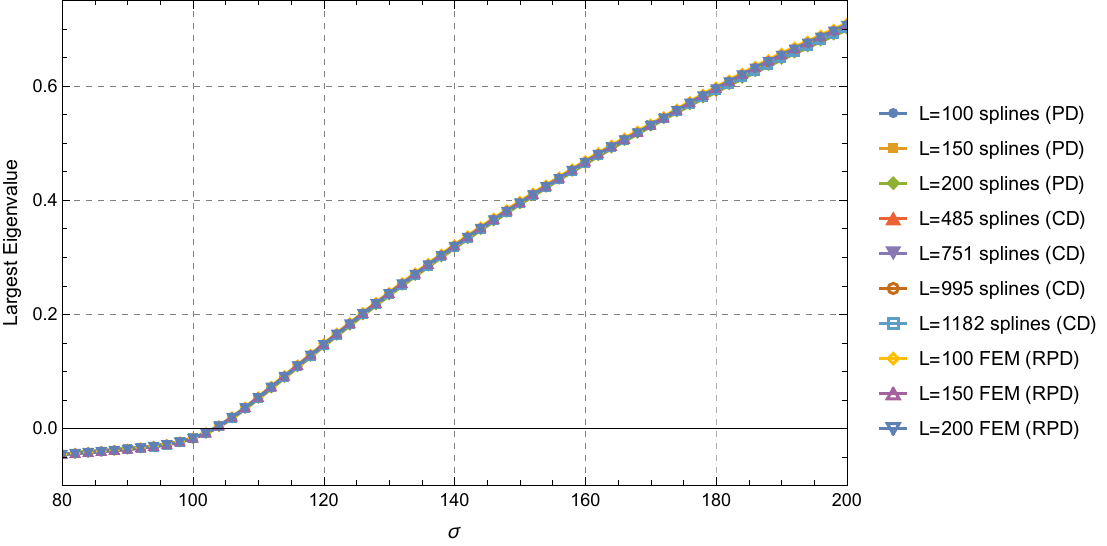}
\caption{The leading eigenvalue during 10 different simulations, using several methods, computed in either the physical domain (PD), the computational domain (CD), or the rescaled physical domain (RPD). The plotted curve actually consists of 10 nearly indistinguishable overlapping curves. The maximum deviation of the leading eigenvalue we have observed among all of these simulations is less than $5\times10^{-3}$, for all sampled values of $\sigma\in[80,200]$.}
\label{fig:eigenvalues}
\end{figure}

More importantly, instability appears to be a broader phenomenon, even in the class of swirl-free vector fields. We have chosen the specific asymptotic data \eqref{sf2} mainly because of its algebraic simplicity, but other homogeneous asymptotic data with suitable parities, such as 
\begin{equation}\label{data3}
\Psi(r,z)=\frac{\sigma z^3}{(r^2+z^2)^2}, \qquad \Omega(r,z)=\frac{\sigma z(6r^2-14 z^2)}{(r^2+z^2)^3},
\end{equation}
also lead to unstable eigenvalues, at different transition thresholds (such as $\sigma\approx 145$ for the asymptotic data \eqref{data3}).

\subsection{Validation} 
Our primary form of validation is demonstrating that all our solutions satisfy the continuous equations to an extremely high degree of precision, as measured on dense uniform grids rather than merely at the collocation points. For comparison, spurious eigenmodes found by the solver exhibit much worse dense residuals, of $\approx 10^{-2}$.

Figure \ref{fig:validation} illustrates the convergence properties of both the nonlinear base solution $(\Psi, \Omega)$ and the principal unstable eigenmode $(\widetilde{\Psi}, \widetilde{\Omega})$, as we track the maximum absolute residual and the eigenvalue deviation $|\lambda - \lambda_{\mathrm{ref}}|$ as a function of the total degrees of freedom (DOF), using B-splines of degrees 6, 8, and 10. The results exhibit clear spectral-like convergence. For instance, transitioning from degree-6 to degree-10 splines at $10^5$ degrees of freedom drops the eigenvalue error by nearly three orders of magnitude. At the highest resolution (approximately $3 \times 10^5$ DOF with degree-10 splines), the maximum pointwise residual of the nonlinear solution drops below $10^{-9}$, and the computed eigenvalue converges to within $10^{-12}$ of our high-precision reference value $\lambda_{\mathrm{ref}}$. This steep, monotonic convergence confirms that our discrete approximations are resolving a true, exact mathematical solution rather than a numerical artifact.

\begin{figure}[ht]
\centering
\includegraphics[width=\textwidth]{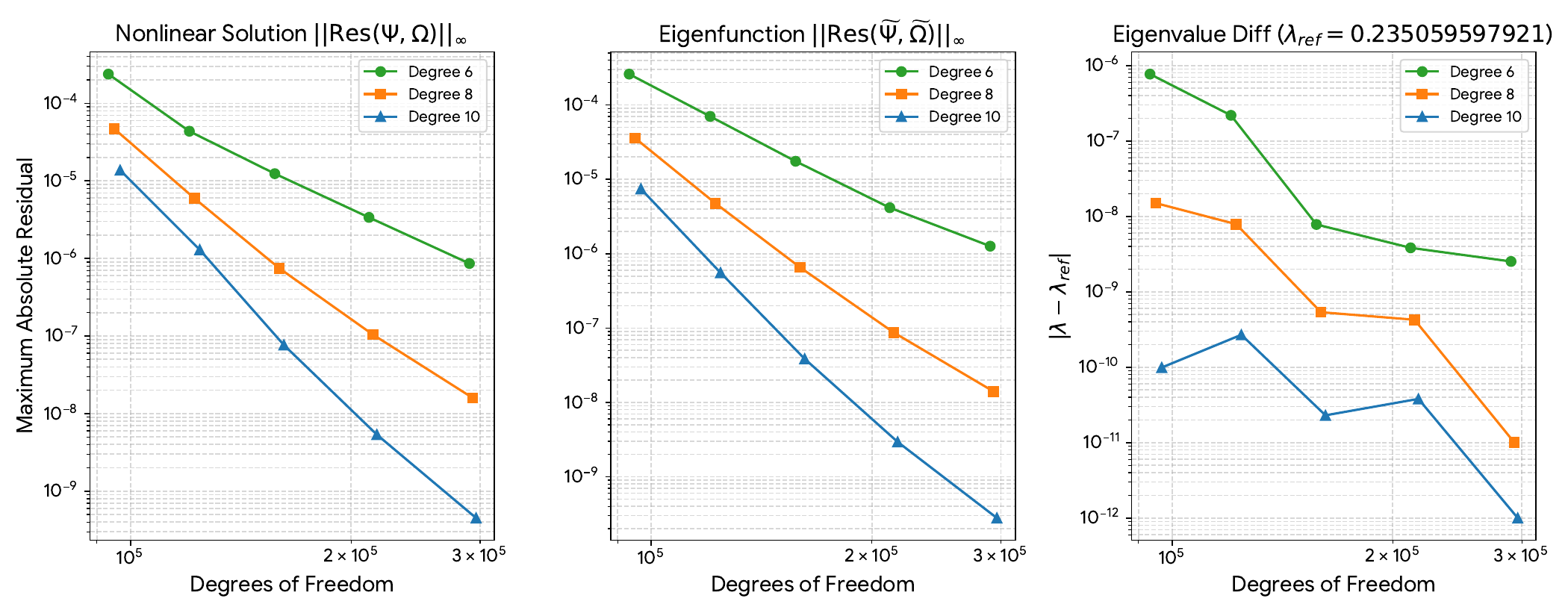}
\caption{Dense residuals and eigenvalue deviation of the refined solutions $(\Psi_2,\Omega_2)$ and $(\widetilde{\Psi}_2, \widetilde{\Omega}_2)$ using degree 6, degree 8, and degree 10 B-splines, with a total number of degrees of freedom ranging from 46,500 DOF to 148,000 DOF per variable. Notice the decrease of the dense residuals as the degrees of the B-splines and the number of degrees of freedom increases.}
\label{fig:validation}
\end{figure}

\subsection{Asymptotic expansion at infinity}\label{numexp} We can gain insight, both at the numerical and at the theoretical level, by expanding at infinity the solutions of the system \eqref{sf1} with asymptotic data \eqref{sf2}. We start with the top order terms, homogeneous of orders $-1$ for $\Psi$ and $-3$ for $\Omega$,
\begin{equation}\label{sf20}
\Psi_{-1}(r,z):=\frac{\sigma z}{r^2+z^2},\qquad \Omega_{-3}(r,z):=\frac{-6\sigma z}{(r^2+z^2)^2}.
\end{equation}
The first equation in \eqref{sf1} is satisfied identically for these functions. The second equation can be used to determine the homogeneous term $\Omega_{-5}$, by the formula
\begin{equation*}
\mathcal{L}\Omega_{-3}+\frac{1}{2}\big(r\partial_r+z\partial_z+3\big)\Omega_{-5}-r\partial_z\Psi_{-1}\partial_r\Omega_{-3}+(r\partial_r\Psi_{-1}+2\Psi_{-1})\partial_z\Omega_{-3}=0,
\end{equation*}
which leads to 
\begin{equation}\label{moreOmega}
\Omega_{-5}=\frac{12\sigma^2z(3z^2-2r^2)+24\sigma z(r^2+z^2)}{(r^2+z^2)^4}.
\end{equation}
We can then use the first equation in \eqref{sf1} to find the term of order $-3$ in $\Psi$, 
\begin{equation}\label{morePsi}
\Psi_{-3}=\frac{\sigma^2z(3r^2-7z^2)}{3(r^2+z^2)^3}-\frac{6\sigma z}{(r^2+z^2)^2}+\frac{\sigma'}{(r^2+z^2)^{3/2}},
\end{equation}
where $\sigma'\in\R$. Since our base solution is odd in $z$ we must have $\sigma'=0$, therefore
\begin{equation}\label{sf21}
\begin{split}
&\Psi(r,z)=\frac{\sigma z}{r^2+z^2}+\frac{\sigma^2z(3r^2-7z^2)}{3(r^2+z^2)^3}-\frac{6\sigma z}{(r^2+z^2)^2}\ldots,\\
&\Omega(r,z)=\frac{-6\sigma z}{(r^2+z^2)^2}+\frac{12\sigma^2z(3z^2-2r^2)+24\sigma z(r^2+z^2)}{(r^2+z^2)^4}\ldots.
\end{split}
\end{equation}
This gives us the refined boundary data \eqref{var13ref} we use to construct our refined solutions. 

Unfortunately, this procedure cannot be continued further. In fact, we could still get a homogeneous term of degree $-7$ in $\Omega$, by repeating the argument above, but this term is too small in the gluing region to matter. However, the equation $\mathcal{L}(\Psi)=0$ admits general homogeneous solutions of degree $-4$ at infinity of the form 
\begin{equation}\label{Psi_-4}
\Psi_{-4}(r,z)=\sigma''z(r^2+z^2)^{-5/2},\qquad\sigma''\in\R.
\end{equation}
These solutions are odd in $z$, thus acceptable. The choice of the parameter $\sigma''$ is determined by global smoothness considerations and cannot be found just by asymptotic analysis.

We use the refined expansion \eqref{sf21}, instead of the more basic expansion \eqref{sf2} in the second stage of our construction, mainly to get better estimates on the residual $|\mathrm{Res}(\Psi_2^\ast)|$ in the gluing region, which leads to better overall estimates on the global residual $\|\mathrm{Res}(\Psi)\|_{L^\infty}$ in \eqref{MainNum2}.

\subsubsection{Further asymptotics}\label{asym_extra} One can further improve the gluing residuals by using more precise asymptotic expansions of the functions $\Psi,\Omega$, with coefficients determined numerically. For example, we can allow a general homogeneous term of the form \eqref{Psi_-4} and optimize over the choice of parameter $\sigma''$. By including this term in our asymptotic function $\Psi_2^a$ and numerically minimizing the far-field difference $|\Psi_2-\Psi_2^a|$ over $\sigma''$, we find an optimal value of $\sigma'' \approx 276,500$. Rerunning the refinement step with this improved choice of 
\begin{equation*}
\Psi_2^a(r,z)=\frac{\sigma z}{r^2+z^2}+\frac{\sigma^2z(3r^2-7z^2)}{3(r^2+z^2)^3}-\frac{6\sigma z}{(r^2+z^2)^2}+\frac{\sigma''z}{(r^2+z^2)^{5/2}},
\end{equation*}
and making the corresponding changes in the boundary data \eqref{var13ref} ultimately reduces the value of $\|\mathrm{Res}(\Psi)\|_{L^\infty}$ in \eqref{MainNum2} by a factor of about $4$. This procedure could, in principle, be continued with more asymptotic terms, to reduce the gluing residuals of the variables.


\begin{thebibliography}{100}

\bibitem{AlBrCo} D. Albritton, E. Bru\'{e}, and M. Colombo, Non-uniqueness of Leray solutions of the forced
Navier--Stokes equations. Ann. of Math. (2) {\bf{196}} (2022), no. 1, 415--455.

\bibitem{AlGuKoRe} D. Albritton, J. Guillod, M. Korobkov, and X. Ren, Forward self-similar solutions to the 2D Navier--Stokes equations. Preprint (2026), arXiv:2601.03161.

\bibitem{BuVi} T. Buckmaster and V. Vicol, Nonuniqueness of weak solutions to the Navier--Stokes equation. 
Ann. of Math. (2) {\bf{189}} (2019), no. 1, 101--144.

\bibitem{ChenHou2022} J. Chen and T. Y. Hou, Stable nearly self-similar blowup of the 2D Boussinesq and 3D Euler equations with smooth data I: Analysis. Preprint (2022), arXiv:2210.07191.

\bibitem{ChenHou2023} J. Chen and T. Y. Hou, Stable nearly self-similar blowup of the 2D Boussinesq and 3D Euler equations with smooth data II: Rigorous Numerics. Preprint (2023), arXiv:2305.05660.

\bibitem{ChDaPa} A. Cheskidov, M. Dai, and S. Palasek, Instantaneous Type I blow-up and non-uniqueness of smooth solutions of the Navier--Stokes equations. Preprint (2025), arXiv:2511.09556.

\bibitem{ChLu1} A. Cheskidov and X. Luo. Sharp nonuniqueness for the Navier–Stokes equations. Invent. Math. {\bf{229}} (2022), no. 3, 987-1054.

\bibitem{ChLu2} A. Cheskidov and X. Luo. $L^2$-critical nonuniqueness for the 2D Navier--Stokes equations. Ann. PDE {\bf{9}} (2023), no. 2, Paper No. 13, 56 pp.

\bibitem{CoPa} M. Coiculescu and S. Palasek, Non-uniqueness of smooth solutions of the Navier--Stokes equations from critical data. Invent. Math. (to appear).

\bibitem{GaWa} T. Gallay and C. E. Wayne,  Invariant manifolds and the long-time asymptotics of the Navier--Stokes and vorticity equations on $\R^2$, Archive for Rational Mechanics and Analysis 163 (3), 209-258

\bibitem{GaSv} T. Gallay and V. \v Sver\'{a}k, Remarks on the Cauchy problem for the axisymmetric Navier--Stokes equations. 
Confluentes Math. {\bf{7}} (2015), no. 2, 67--92.

\bibitem{GuLiXi} C. Gui, H. Liu, and C. Xie, On the forward self-similar solutions to the two-dimensional Navier--Stokes equations. Preprint (2026), arXiv:2601.03833.

\bibitem{GuSv} J. Guillod and V. \v Sver\'{a}k, Numerical investigations of non-uniqueness for the Navier--Stokes initial value problem in borderline spaces. J. Math. Fluid Mech. {\bf{25}} (2023), no. 3, Paper No. 46, 25 pp.

\bibitem{HoWaYa} T. Hou, Y. Wang, and C. Yang, Nonuniqueness of Leray-Hopf solutions to the unforced incompressible 3D Navier--Stokes Equation. Preprint (2025), arXiv:2509.25116.

\bibitem{JiSv1} H. Jia and V. \v Sver\'{a}k, Local-in-space estimates near initial time for weak solutions of the Navier--Stokes equations and forward self-similar solutions. Invent. Math. {\bf{196}} (2014), no. 1, 233--265.

\bibitem{JiSv2} H. Jia and V. \v Sver\'{a}k, Are the incompressible 3d Navier--Stokes equations locally ill-posed in the natural energy space? J. Funct. Anal. {\bf{268}} (2015), no. 12, 3734--3766.

\bibitem{La} O. Ladyzhenskaya, Unique solvability in the large of the three-dimensional Cauchy problem
for the Navier--Stokes equations in the presence of axial symmetry. Zap. Nauchn. Semin.
Leningr. Otd. Mat. Inst. Steklova {\bf{7}} (1968), 155--177 (in Russian).

\bibitem{Leray} J. Leray, Sur le mouvement d'un liquide visqueux emplissant l'espace, Acta Math. {\bf 63} (1934), no.~1, 193--248.

\bibitem{UkYu} M. Ukhovskii and V. Yudovich, Axially symmetric flows of ideal and viscous fluids filling
the whole space. J. Appl. Math. Mech. {\bf{32}} (1968), 52--61.

\bibitem{PINN} Y. Wang, M. Bennani, J. Martens, S. Racani\'{e}re, S. Blackwell, A. Matthews, S. Nikolov, G. Cao-Labora, D. S. Park, M. Arjovsky, D. Worrall, C. Qin, F. Alet, B. Kozlovskii, N. Tomašev, A. Davies, P. Kohli, T. Buckmaster, B. Georgiev, J. G\'{o}mez-Serrano, R. Jiang, C.-Y. Lai, Discovery of unstable singularities. Preprint (2025), arXiv:2509.14185.


\end{thebibliography}
\end{document}